\documentclass[12pt,reqno]{amsart}

\AtBeginDocument{\addtocontents{toc}{\protect\setlength{\parskip}{0pt}}}

\usepackage{latexsym,amsmath,amssymb,amsthm,mathtools,mathabx}
\usepackage{enumitem}
\usepackage{graphicx}
\usepackage{hyperref}
\hypersetup{
  colorlinks=false,
  pdfborder={0 0 1},
  linkbordercolor={1 0 0},
  citebordercolor={1 0 0},
  urlbordercolor={1 0 0},
  bookmarksnumbered=true
}

\numberwithin{equation}{section}

\theoremstyle{plain}
\newtheorem{theorem}{Theorem}[section]
\newtheorem{lemma}[theorem]{Lemma}
\newtheorem{proposition}[theorem]{Proposition}

\theoremstyle{definition}
\newtheorem{definition}[theorem]{Definition}
\newtheorem{example}[theorem]{Example}

\theoremstyle{remark}
\newtheorem{remark}[theorem]{Remark}

\newcommand{\R}{\mathbb R}
\newcommand{\Z}{\mathbb Z}
\newcommand{\C}{\mathbb C}
\newcommand{\N}{\mathbb N}
\newcommand{\Sch}{\mathfrak S}
\newcommand{\Tr}{\operatorname{Tr}}

\newcommand{\supp}{\operatorname{supp}}

\newcommand{\rad}{\mathrm{rad}}
\newcommand{\KP}{\mathrm{KP}}
\newcommand{\ZK}{\mathrm{ZK}}

\newcommand{\ip}[2]{\left\langle #1,#2\right\rangle}

\newcommand{\dd}{\,d}
\newcommand{\norm}[1]{\left\|#1\right\|}

\title[Almost-orthogonal Strichartz estimates]{Almost-orthogonal Strichartz estimates \\ and radial improvements}

\author{Hongzhou Ji, Liping Xu, and An Zhang}

\address{School of Mathematical Sciences, Beihang University, 37 Xueyuan Road, Beijing, 100191, China}
\email{jhz2509@buaa.edu.cn}
\email{xuliping.p6@gmail.com}
\email{anzhang@buaa.edu.cn}

\begin{document}

\begin{abstract}
We extend the Schatten-duality principle from orthonormal systems to (almost-orthogonal) Bessel families and apply it to prove sharp and improved Bessel-family Strichartz estimates for the Kadomtsev--Petviashvili, Zakharov--Kuznetsov, and radial Schr\"odinger equations. In particular, for the Schr\"odinger equation, we establish an improved estimate in a larger radial range of space-time exponents, not from dispersive estimates as in the KP/ZK models, but from direct robust Schatten estimates, using some vector-valued multi-product trace formula and multilinear weighted fractional integral formula. To the best of our knowledge, this is the first radial improvement for systems, and the almost-orthogonal estimates are also new to the literature. We also derive necessary conditions.
\end{abstract}

\date{\today. ~\emph{MSC.} 42B37, 58J50, 35P05. \emph{Keywords.} Strichartz estimates, almost-orthogonal systems, Bessel family,  radial improvements, dispersive equations, Schr\"odinger equations, Kadomtsev--Petviashvili equations, Zakharov--Kuznetsov equations.}

\maketitle

\tableofcontents

\section{Introduction}\label{sec:introduction}

Let $X$ be a spatial measure space, and let $(U(t))_{t\in\R}$ be the
propagator associated with a linear dispersive equation on $X$.  The classical
(single-function) Strichartz estimate, introduced by Strichartz
\cite{Strichartz1977} and further developed by Keel--Tao \cite{KeelTao}, states
that for suitable exponents $(p,q)$,
\begin{equation}\label{eq:intro-single}
  \norm{U(t)f}_{L_t^{2p}(\R;L_x^{2q}(X))}\lesssim \norm{f}_{L^2(X)}.
\end{equation}
For a system $(f_j)_j\subset L^2(X)$ with proper orthogonality, we consider the following system 
Strichartz estimate:
\begin{equation}\label{eq:sys-str}
  \norm{\sum_j\lambda_j|U(t)f_j|^2}_{L_t^p(\R;L_x^q(X))}
  \lesssim \norm{\lambda}_{\ell^\beta}.
\end{equation}
The Bourgain notation $\lesssim$ will be used throughout the paper, with implicit constants. We call $\beta$ the discrete \emph{system exponent} and aim to find the largest $\beta$ for \eqref{eq:sys-str} to hold with fixed $(p,q)$.  Applying \eqref{eq:intro-single} with the
triangle inequality yields the
$\ell^1$-bound ($\beta=1$) for \eqref{eq:sys-str}, whereas an estimate with $\beta>1$ reflects an
improvement stemming from orthogonality. 

Strichartz estimates for orthonormal systems were introduced by Frank--Lewin--Lieb--Seiringer in their pioneering work \cite{FLLS}
and further developed by Frank--Sabin \cite{FrankSabin} via complex interpolation and Schatten
duality; see e.g.,  Bez-Hong-Lee-Nakamura-Sawano \cite{MR3985036} and Bez--Lee--Nakamura \cite{BLN} for substantial extensions and generalizations in the literature.
The theory has also been developed beyond the Euclidean setting.
Nakamura \cite{NakamuraTorus} established orthonormal
Strichartz estimates for the Schr\"odinger equation on the torus.
Wang--Zhang--Zhang \cite{WZZ} and Ji--Xu--Zhang \cite{JXZ} subsequently improved such results by decoupling and
extended them to compact manifolds for the wave, Klein--Gordon, and fractional
Schr\"odinger equations.
Bhimani--Choudhary \cite{BC} considered fractional
Schr\"odinger equations on tori and waveguide manifolds.
Besides, orthonormal spectral cluster bounds on compact manifolds have also been initiated by Frank--Sabin \cite{FrankSabinSpectral},
and further developed by Ren--Zhang \cite{RenZhang} and Cuenin--Nguyen--Su \cite{CueninNguyenSu} on flat tori and manifolds with nonpositive curvature.
For other generalizations, Hoshiya
\cite{Hoshiya2024,Hoshiya2025} considered dispersive equations with potentials and for
repulsive Hamiltonians, while Feng--Mondal--Song--Wu
\cite{FMSW} investigated
nonnegative self-adjoint operators on general measure spaces.
Recently, Bennett--Bez--Guti\'errez--Nakamura--Oliveira \cite{BBGNO}
developed a new approach based on Wigner distributions and obtained weighted
orthonormal estimates for the Schr\"odinger equation.

The purpose of this paper is to extend the framework to almost-orthogonal Bessel families and to apply it to the
Kadomtsev--Petviashvili, Zakharov--Kuznetsov, and radial Schr\"odinger
propagators.  To the best of our knowledge, no improved system Strichartz estimate has been established before for a larger radial range, even for orthonormal systems.

\subsection{Basic definitions}\label{subsec:intro-definitions}

For a Hilbert space $(H,\langle\cdot,\cdot\rangle)$ and $1\le r<\infty$, the Schatten class
$\Sch^r(H)$ consists of compact operators $A$ whose singular values
$(s_n(A))_n$ belong to $\ell^r$, endowed with the norm
\[
  \norm{A}_{\Sch^r(H)}
  =\left(\sum_n s_n(A)^r\right)^{1/r}.
\]
We use the standard convention
\[
  \norm{A}_{\Sch^\infty(H)}=\norm{A}_{H\to H}.
\]
\begin{definition}\label{def:intro-bessel}
Let $H$ be a Hilbert space.  A normalized family $f=(f_j)_j\subset H$
is called a \emph{Bessel family} if there exists a constant $C<\infty$ such that
\begin{equation}\label{eq:intro-bessel}
  \norm{\sum_j \lambda_j f_j}_{H}^{2}
  \le C\sum_j |\lambda_j|^{2}
\end{equation}
for every finitely supported sequence $\lambda=(\lambda_j)_j$.  The smallest constant $C$ for
which \eqref{eq:intro-bessel} holds uniformly for all such $\lambda$ is denoted by $C_B=C_B(f)$ and is called the
\emph{Bessel bound} of the family $f$.
\end{definition}

\begin{example}\label{ex:intro-bessel-examples}
To illustrate the notion, we provide two examples of Bessel families.
\begin{enumerate}[label=\textup{(\roman*)}]
\item Let $(g_k)_{k\ge1}$ be an orthonormal family and fix $m\in\N^*$.  Then $(g_k)_{k\ge1}$ is a Bessel family with bound $C_B=1$.   If every
$g_k$ is repeated $m$ times, namely
\[
  f_{k,\ell}=g_k,\qquad k\ge1,\quad 1\le \ell\le m,
\]
then $(f_{k,\ell})$ is a normalized Bessel family with bound $C_B=m$.

\item Let $(e_j)_{j\ge1}$ be the standard orthonormal basis of $\ell^2$ and fix
$0\le\delta\le1$.  Define
\[
  f_{2k-1}=e_{2k-1},\qquad
  f_{2k}=\delta e_{2k-1}+\sqrt{1-\delta^2}\,e_{2k},
  \qquad k\ge1.
\]
Then $(f_j)_{j\ge1}$ is a Bessel family with bound $C_B=1+\delta \ \to1$ as $\delta\to0$.

\end{enumerate}
\end{example}

\begin{remark}\label{rem:intro-bessel-normalization}
We always have $C_B\ge1$, and $C_B=1$ if and only if the family is
orthonormal.  Thus $C_B$ quantifies the degree of almost orthogonality: the
closer $C_B$ is to $1$, the more nearly orthogonal the family is.  Let
$(e_j)_j$ be the standard orthonormal basis of $\ell^2$.  For a
Bessel family $f=(f_j)_j\subset H$, the operator
\[
  T_f:\ell^2\to H,\qquad T_fe_j=f_j,
\]
is bounded with
\[
  C_B=\norm{T_f}^2_{\ell^2\to H}
   =\norm{T_f^*T_f}_{\ell^2\to\ell^2}.
\]
\end{remark}

We next consider three \emph{dispersive propagators}.  
Throughout the paper, for $d\ge1$,  we use the unitary Fourier transform
\begin{align*}
  \widehat f(\xi)=\mathcal F_x f(\xi)
  :=(2\pi)^{-d/2}\int_{\R^d}e^{-ix\cdot\xi}f(x)\,\dd x, \quad \forall\,\xi\in \mathbb R^d.
\end{align*}
For the
Kadomtsev--Petviashvili equations,
\begin{equation}\label{eq:KPeq}\partial_t u+\partial_x^3u\mp \partial_x^{-1}\partial_y^2u=0, \qquad t\in \mathbb R, \ z=(x,y)\in \mathbb R^2,\end{equation}
 set the phase
\[
  \Phi_{\KP}(\xi,\eta)
  =\xi^3\pm\frac{\eta^2}{\xi},
  \qquad 
  \zeta=(\xi,\eta)\in\mathbb R^2, \ \xi\ne0,
\]
and define the propagator $U_{\KP}(t)$ by
\begin{equation*}
  \widehat{U_{\KP}(t)f}(\xi,\eta)
  =e^{it\Phi_{\KP}(\xi,\eta)}\widehat f(\xi,\eta).
\end{equation*}
For the Zakharov--Kuznetsov equation ($d\ge2$),
\begin{equation}\label{eq:ZKeq}
  \partial_tu+\partial_x\Delta_z u=0,
  \qquad t\in\R,\ z=(x,y)\in\R\times\R^{d-1},
\end{equation}
with Fourier variables $\zeta=(\xi,\eta)\in\R\times\R^{d-1}$, the phase
and propagator are
\[
  \Phi_{\ZK}(\zeta)=\zeta_1|\zeta|^2=\xi(\xi^2+|\eta|^2),
  \qquad
  \widehat{U_{\ZK}(t)f}(\zeta)
  =e^{it\Phi_{\ZK}(\zeta)}\widehat f(\zeta).
\]
Finally, for the (radial) Schr\"odinger equation
\begin{equation}\label{eq:Scheq}
  i\partial_tu+\Delta_xu=0,
  \qquad t\in\R,\ x\in\R^d,
\end{equation}
the propagator is $U_S(t)=e^{it\Delta}$, so that
\[
  \widehat{U_S(t)f}(\xi)=e^{-it|\xi|^2}\widehat f(\xi),
  \qquad \xi\in\R^d.
\]

We will restrict the Schr\"odinger propagator to
radial initial data, which will admit a larger (radial) range of admissible exponents for Strichartz estimates to hold; see Theorem~\ref{thm:intro-radial} and \ref{thm:intro-known-radial-single}.
The radial subspace of $L^2(\mathbb R^d)$ is
\[
  L^2_{\rad}(\R^d)
  =\{f\in L^2(\R^d):f(x)=F(|x|)\text{ for some function }F\}.
\]
A \emph{radial Bessel family} is a Bessel family whose elements lie in
$L^2_{\rad}(\R^d)$.

For $N>0$, let $P_N$ denote the smooth annular Fourier multiplier
\[
  \widehat{P_Nf}(\zeta)=\psi(\zeta/N)\widehat f(\zeta),
\]
where $\psi\in C_c^\infty(\R^d)$ is a real-valued radial function such that 
\[
 0\le\psi\le1, \qquad \supp\psi\subset\{\zeta:1/2<|\zeta|<2\},
  \qquad
  \psi(\zeta)=1 \ \text{when } \  \frac34\le|\zeta|\le\frac43.
\]
Since $\psi$ is radial, we also write
\begin{equation}\label{eq:cutoff}
  \psi(\zeta)=\chi(|\zeta|),
  \qquad \chi\in C_c^\infty((1/2,2)).
\end{equation}

\subsection{Main results}\label{subsec:intro-main-results}
We now state the  Bessel-family Strichartz estimates for above-mentioned three dispersive models.  
Define the (sharp) system exponent function
\begin{equation}\label{eq:betaq}
  \beta(q):=\frac{2q}{q+1},
  \qquad 1\le q<\infty.
\end{equation}

Besides the typical (fractional) Schr\"odinger, wave and Klein-Gordon equations, we want to consider some different models.
The first theorem is for the Kadomtsev--Petviashvili equations \eqref{eq:KPeq}.
\begin{theorem}[KP]\label{thm:intro-KP}
Let $1\le q<\infty,\, 1<p\le \infty$ be such that 
\[
  \frac1p+\frac1q=1.
\]
Suppose that \(\beta\ge 1\) satisfies
\begin{equation}\label{eq:KP-beta-range}
  \begin{cases}
    \beta \le \beta(q), & 1\le q<3,\\[1.2ex]
    \beta < p, & 3\le q<\infty.
  \end{cases}
\end{equation}
Then for every sequence $\lambda=(\lambda_j)_j\in\ell^\beta$ and every Bessel family $f=(f_j)_j\subset L^2(\R^2)$,
\begin{equation}\label{eq:global-KP-estimate}
  \norm{\sum_j\lambda_j|U_{\KP}(t)f_j|^2}_{L_t^p(\R;L_{z}^q(\R^2))}
  \lesssim C_B(f)^{1-1/\beta}\norm{\lambda}_{\ell^\beta}.
\end{equation}
\end{theorem}

\IfFileExists{KPZK.pdf}{
\begin{figure}[ht]
  \centering
  \includegraphics[width=1.0\textwidth]{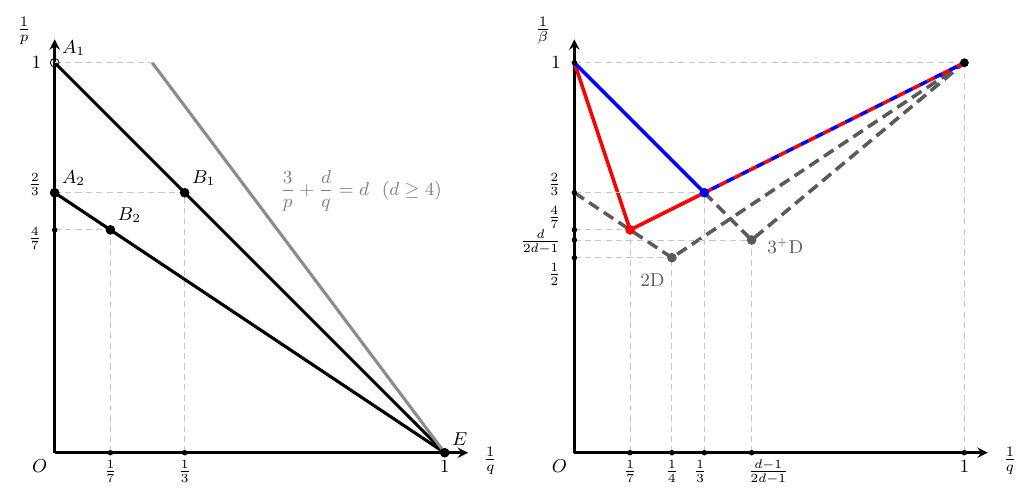}
  \caption{The left panel shows the admissible region for the KP and ZK models. 
  In the right panel, the red and blue lines indicate the sufficient conditions for the $2$D and $3^{\scriptscriptstyle +}$D-ZK models, respectively,
  while the gray dashed lines indicate the (unproved) necessary conditions.
  The blue line also represents the sharp system exponent for the KP model.}
  \label{fig:KP-admissible}
\end{figure}
}{}
\begin{remark}[Sharpness]\label{rem:intro-KP-necessary}
Proposition~\ref{prop:KP-ZK-necessary} \textup{(KP)} shows that for \eqref{eq:global-KP-estimate}
 to hold, it is necessary that 
\[
  \beta\le \beta(q)\quad(1\le q\le3),
  \qquad
  \beta\le p\quad(3\le q<\infty).
\]
Thus, in the subcritical range $1\le q<3$, which corresponds to the segment $(B_1, E]$ in Figure~\ref{fig:KP-admissible}, the system exponent range \eqref{eq:KP-beta-range}
is already sharp.  In the supercritical range $3\le q<\infty$, which corresponds to the segment $[B_1, A_1)$ in Figure~\ref{fig:KP-admissible}, the range \eqref{eq:KP-beta-range} is sharp in the sense
that \eqref{eq:global-KP-estimate} fails whenever $\beta>p$.
\end{remark}

The next theorem is for the Zakharov--Kuznetsov equations \eqref{eq:ZKeq}. 
\begin{theorem}[ZK]\label{thm:intro-ZK}
\emph{(i)} Let $d=2$. Let $1\le q\le\infty,\, 3/2\le p\le \infty$ be such that
\[
  \frac3p+\frac2q=2.
\]
Suppose that \(\beta\ge 1\) satisfies
\begin{equation}\label{eq:2d-ZK-beta-range}
  \begin{cases}
    \beta \le \beta(q), & 1\le q<7,\\[1.2ex]
    \beta < q/(q-3), & 7\le q<\infty,\\[1.2ex]
    \beta=1, & q=\infty.
  \end{cases}
\end{equation}
Then for every sequence $\lambda=(\lambda_j)_j\in\ell^\beta$ and every Bessel family $f=(f_j)_j\subset L^2(\R^2)$,
\begin{equation}\label{eq:2d-zk-strichartz}
  \norm{\sum_j\lambda_j|U_{\ZK}(t)f_j|^2}_{L_t^p(\R;L_z^q(\R^2))}
  \lesssim C_B(f)^{1-1/\beta}\norm{\lambda}_{\ell^\beta}.
\end{equation}
\emph{(ii)} Let $d\ge3$. 
Let $p,q,\beta$ satisfy the same condition as \emph{Theorem~\ref{thm:intro-KP}}. 
Then for every frequency $N>0$, every sequence $\lambda=(\lambda_j)_j\in\ell^\beta$ and every Bessel family $f=(f_j)_j\subset L^2(\R^d)$,
\begin{equation}\label{eq:higer-zk-strichartz}
  \norm{\sum_j\lambda_j|U_{\ZK}(t)P_Nf_j|^2}_{L_t^p(\R;L_z^q(\R^d))}
  \lesssim C_B(f)^{1-1/\beta}
  N^{(d-3)(1-1/q)}\norm{\lambda}_{\ell^\beta}.
\end{equation}
\end{theorem}

\begin{remark}[Sharpness]\label{rem:intro-ZK-necessary}
Proposition~\ref{prop:KP-ZK-necessary}~\textup{($3^{\scriptscriptstyle +}$D-ZK)} shows that the range \eqref{eq:KP-beta-range} in Theorem~\ref{thm:intro-ZK}~\textup{(ii)} is sharp for $3\le q<\infty$ and for $q=1$ (the segment $(A_1,B_1]$ and the point $E$ in Figure~\ref{fig:KP-admissible}).
However, sharpness in the regime $1<q<3$ (the segment $(B_1,E)$) remains unknown, since the necessary condition in Proposition~\ref{prop:KP-ZK-necessary}~\textup{($3^{\scriptscriptstyle +}$D-ZK)} leaves a gap.
For $d=2$, the sharpness of the range \eqref{eq:2d-ZK-beta-range} for \eqref{eq:2d-zk-strichartz} remains unknown, except for $(p,q)=(7/4,7)$ and $(p,q)=(\infty,1)$ (the points $B_2$ and $E$), in view of the necessary condition in Proposition~\ref{prop:KP-ZK-necessary}~\textup{(2D-ZK)}.
\end{remark}

\begin{remark}[Global estimate]
\label{rem:intro-ZK-frequency-global}
For $d\ge4$, the localization restriction of \eqref{eq:higer-zk-strichartz} in Theorem~\ref{thm:intro-ZK}~\textup{(ii)} can
be removed at the cost of Sobolev regularity, by the vector-valued
Littlewood--Paley inequality (\cite[Lemma 1]{sabin16}). 
The global estimate
\begin{equation*}
  \norm{\sum_j\lambda_j|U_{\ZK}(t)f_j|^2}_{L_t^p(\R;L_z^q(\R^d))}
  \lesssim C_B(f)^{1-1/\beta}\norm{\lambda}_{\ell^\beta}
\end{equation*} holds for every Bessel family $ f=(f_j)_j\subset H^s(\R^d)$ with
\[s>(d-3)(1-1/q)/2.
\] 

\end{remark}

For radial improvement for dispersive equations, we consider the radial Schr\"odinger model \eqref{eq:Scheq}. 
Let $d\ge2$ and set
\begin{equation}\label{eq:intro-radial-parameters}
  q_* = \frac{2d+1}{2d-1},
  \qquad
  q_c=\frac{2d-1}{2d-3}.
\end{equation}

\begin{theorem}[Radial Schr\"odinger]
\label{thm:intro-radial}
Let $d\ge2$.  Let $p,q$ satisfy
\begin{equation}\label{eq:radial-boundary-segment}
  1\le q<q_c\,,
  \qquad
  \frac1p+\frac{2d-1}{2q}=d-\frac12\,.
\end{equation}
Suppose that \(\beta\ge 1\) satisfies
\begin{equation}\label{eq:radial-beta-range}
  \begin{cases}
    \beta \le \beta(q), & 1\le q\le q_*\,,\\[1.2ex]
    \beta < 2p/(p+1), & q_*< q<q_c\,.  \end{cases}
\end{equation}
Then for every frequency $N>0$, every sequence $\lambda=(\lambda_j)_j\in\ell^\beta$ and every radial Bessel family  $f=(f_j)_j\subset L^2_{\rad}(\R^d)$, 
\begin{equation}\label{eq:radial-boundary-estimate}
  \norm{\sum_j\lambda_j|U_S(t)P_Nf_j|^2}_{L_t^p(\R;L_x^q(\R^d))}
  \lesssim C_B(f)^{1-1/\beta}
  N^{-(d-1)(1-1/q)}
  \norm{\lambda}_{\ell^\beta}.
\end{equation}

\end{theorem}

\IfFileExists{radial_combined.pdf}{
\begin{figure}[ht]
  \centering
  \includegraphics[width=1.0\textwidth]{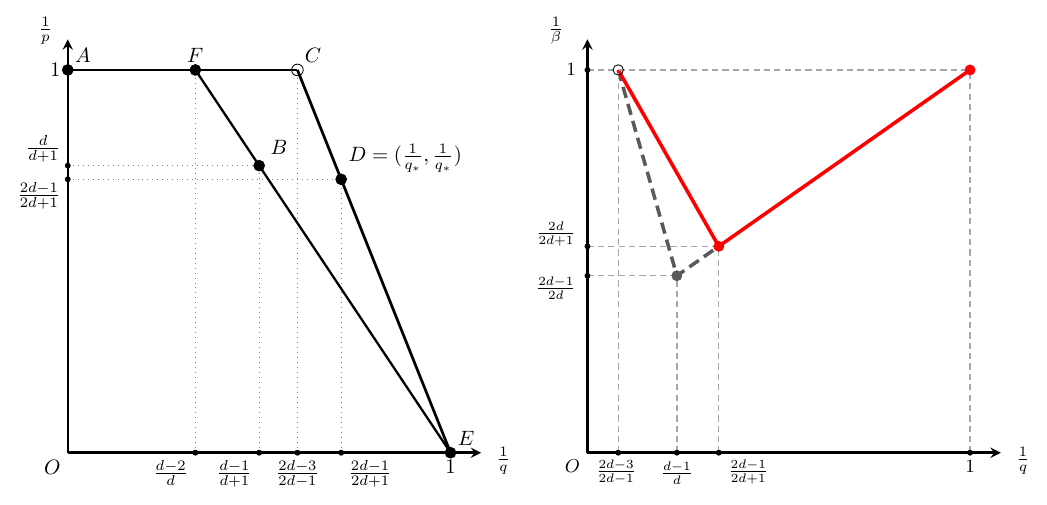}
  \caption{The left panel shows the admissible region for the radial Schr\"odinger model. 
  In the right panel, the red line represents the sufficient condition for the radial Schr\"odinger model, while the gray dashed line represents the (unproved) necessary condition.}
  \label{fig:radial-admissible}
\end{figure}
}{}

\begin{remark}[Sharpness]
\label{rem:intro-radial-necessary}
Proposition~\ref{prop:radial-necessary} shows that valid
$\beta$ for 
\eqref{eq:radial-boundary-estimate} to hold satisfy
\[
  \frac1\beta\ge \frac1p+\frac d q-(d-1), 
\]
which, on the full boundary \eqref{eq:radial-boundary-segment}, reduces to
$\beta\le\beta(q)$.  Hence the range
\eqref{eq:radial-beta-range} is sharp for $1\le q\le q_*$ whereas sharpness for $q_*<q<q_c$  remains open; this corresponds to segments $[D,E]$ and $(D,C)$, respectively, in Figure~\ref{fig:radial-admissible}.
\end{remark}

\begin{remark}
The radial improvement can be generalized to a class of dispersive models, e.g., fractional Schr\"odinger equations; see \cite{GuoWangRadial} for single-function result. However, we cannot get global estimate from the localized one, due to the negative power in frequency. 
\end{remark}

\begin{remark}
We may consider the Bessel-family Strichartz estimates with extra regularity of initial data or extra power of frequency over the full admissible regions for all three aforementioned and more other models; see e.g., \cite{MR3985036,JXZ}. Here we list only results on the sharp boundary lines, leaving the non-sharp regions as exercises.
\end{remark}

\subsection{Classical single-function results}
\label{subsec:known-single-function-estimates}

In this subsection, we review some  classical single-function Strichartz estimates for the three dispersive models in the whole admissible regions.

The analysis of KP equations by Fourier restriction methods goes back, 
in particular, to Bourgain's work \cite{BourgainKP} on the periodic KP-II equation. 
For the Euclidean setting, Hadac \cite{HadacKP} proved Strichartz estimates for dispersion-generalized KP-II equations.
For the KP-I equation, Molinet--Saut--Tzvetkov
\cite{MSTBackgroundKPI} obtained Strichartz estimates in the study of
KP-I solutions with nonlocalized backgrounds.
More recently, Guo--Molinet \cite{GuoMolinetKPI} obtained refined Strichartz estimates for KP-I solutions.
Linear Strichartz estimates for other KP-type models can be found, for
example, in \cite{PilodSautSelbergTesfahun,KinoshitaSanwalSchippa}.  
The following theorem combines the KP-I
estimate from \cite[Lemma~4]{MSTBackgroundKPI} and the KP-II estimate from
\cite[Theorem~3.1]{HadacKP}.

\begin{theorem}
\label{thm:intro-known-KP-single}
Let $1\le p\le \infty,\, 1\le q<\infty$ such that $1/p+1/q\le1$. Then for every $f\in \dot H^{\sigma_{\KP}(p,q)}(\mathbb R^2)$ with $\sigma_{\KP}(p,q):=1-1/p-1/q$, 
\begin{equation*}
  \norm{U_{\KP}(t)f}_{
    L_t^{2p}(\R;L_{x,y}^{2q}(\R^2))}
  \lesssim \norm{f}_{\dot H^{\sigma_{\KP}(p,q)}(\R^2)}.
\end{equation*}
\end{theorem}

For the Zakharov--Kuznetsov equation in two dimensions, 
the general dispersive theory of Ben-Artzi--Koch--Saut \cite{BenArtziKochSaut} 
for third-order equations applies to the linear ZK propagator. 
Linares--Pastor \cite{LinaresPastor2009,LinaresPastor2011} obtained Strichartz estimates 
in the study of generalized ZK equations, while bilinear refinements were 
subsequently developed by Gr\"unrock--Herr \cite{GrunrockHerrZK} and Molinet--Pilod \cite{MolinetPilod}, 
and further sharpened by Kinoshita \cite{KinoshitaZK}. 
Shan--Wang--Zhang \cite{ShanWangZhangZK} also used bilinear
Strichartz estimates in the study of global well-posedness for the
two-dimensional ZK equation.
In higher dimensions, Herr--Kinoshita \cite{HerrKinoshitaZK} obtained endpoint Strichartz estimates, 
while Linares--Ramos \cite{LinaresRamosZK} established a family of Strichartz estimates for the ZK propagator.

\begin{theorem}
\label{thm:intro-known-ZK-single}
Let $p, q \ge 1$ such that
\begin{equation*}
  \begin{cases}
  3/p+2/q\le2,  \quad  (p,q)\neq (\infty,\infty), & \quad d=2\,,\\
   1/p+1/q\le1, \quad  q<\infty, & \quad d\ge 3\,.  \end{cases}
\end{equation*} 
Then for every $f\in H^{\sigma_{\ZK}(p,q)}(\mathbb R^d)$ with $\sigma_{\ZK}(p,q):=\left[d(1-1/q)-3/p\right]/2$,
\begin{equation*}
  \norm{U_{\ZK}(t)f}_{
    L_t^{2p}(\R;L_z^{2q}(\R^d))}
  \lesssim
  \norm{f}_{H^{\sigma_{\ZK}(p,q)}(\R^d)}.
\end{equation*}

\end{theorem}

For $d=2$, Theorem~\ref{thm:intro-known-ZK-single} follows from \cite[Lemma~2.1]{LinaresPastorDrumondSilva} together with Sobolev embedding. 
For $d\ge3$, Theorem~\ref{thm:intro-known-ZK-single} follows from \cite[Proposition~1.3]{LinaresRamosZK}. 
Although \cite[Proposition~1.3]{LinaresRamosZK} is stated on the fixed time interval $[0,1]$, 
its proof can apply on the whole time axis. Hence the same estimate holds globally in time.

For the Schr\"odinger equation, the classical Strichartz estimates go
back to Strichartz \cite{Strichartz1977}, and the endpoint estimate was
established by Keel--Tao \cite{KeelTao}.  Under radial symmetry, Shao
\cite{ShaoRadial} proved sharp adjoint restriction estimates for
cylindrically symmetric functions on the paraboloid.  Guo--Wang
\cite{GuoWangRadial} obtained the radial Strichartz range
for a class of dispersive equations including the Schr\"odinger
case. Subsequently, Ke \cite{KeRadial} improved upon the result of Guo--Wang. 

\begin{theorem}
\label{thm:intro-known-radial-single}
Let  $p,q\ge1$ such that $(p,q)\neq (1,q_c)$ and 
\[\frac1p+\frac{2d-1}{2q}\le d-\frac12.\]
Then for every $N>0$ and every $f\in L^2_{\rad}(\mathbb R^d)$
\begin{equation*}
  \norm{U_S(t)P_Nf}_{
    L_t^{2p}(\R;L_x^{2q}(\R^d))}
  \lesssim
  N^{\frac12(d-\frac2p-\frac d q)}\norm{f}_{L^2(\R^d)}.
\end{equation*}
\end{theorem}

Theorem~\ref{thm:intro-known-radial-single} follows from
\cite[Theorem~1.1]{KeRadial} and \cite[Theorem~1.5]{GuoWangRadial}, while the sharpness of the radial boundary
was explained in  \cite{GuoWangRadial}.

\subsection{Organization of the paper}

Section~\ref{sec:abstract} proves the general Bessel-family duality principle and 
the dispersive-to-Strichartz framework.  
Section~\ref{sec:KP} applies this
framework to the Kadomtsev--Petviashvili and Zakharov--Kuznetsov equations, 
and establishes the corresponding necessary conditions on the system exponents.
Finally, Section~\ref{sec:radial} treats the radial Schr\"odinger case by directly 
proving a robust Schatten estimate at the extended diagonal endpoint,
followed by interpolation and necessary conditions.

\section{Schatten duality and dispersive transference for Bessel families}\label{sec:abstract}

\subsection{Duality principle for Bessel families}
\label{subsec:almost-extension}

Throughout this subsection, $X$ is a $\sigma$-finite measure space.  Let
$E$ be a linear operator on $L^2(X)$ such that $Ef$ is measurable on
$\R\times X$ for every $f\in L^2(X)$.  For a test measurable function $V$ on
$\R\times X$, define $A_V$ to be a linear operator on $L^2(X)$ such that the following integral is valid for $f,g\in L^2(X)$ \begin{equation*}
  \langle A_Vf,g\rangle_{L^2(X)}
  :=\int_{\R\times X}V(t,x)\,Ef(t,x)\,\overline{Eg(t,x)}\,\dd x\dd t.
\end{equation*}

We first extend the Frank--Sabin duality principle \cite[Lemma~3]{FrankSabin} from orthonormal systems
to Bessel families.

\begin{lemma}
\label{thm:intro-abstract}
Let $1<p,q<\infty$ and $1\le\beta<\infty$.  Then the following two statements are equivalent, with the
same constant $C$.
\begin{enumerate}[label=\textup{(\roman*)}]
\item For every 
$V\in L_t^{p'}(\R;L_x^{q'}(X))$,
\begin{equation*}
  \norm{A_V}_{\Sch^{\beta'}(L^2(X))}
  \le C\,\norm{V}_{L_t^{p'}(\R;L_x^{q'}(X))}.
\end{equation*}

\item For every Bessel family $f=(f_j)_j\subset L^2(X)$ and every sequence $\lambda=(\lambda_j)_j\in\ell^\beta$,
\begin{equation}\label{eq:intro-bessel-conclusion}
  \norm{\sum_j\lambda_j|Ef_j|^2}_{L_t^p(\R;L_x^q(X))}
  \le C\,C_B(f)^{1-1/\beta}\,\norm{\lambda}_{\ell^\beta}.
\end{equation}
\end{enumerate}
\end{lemma}

\begin{proof}
We use the Dirac-notation $|f\rangle\langle g|$ for the rank-one operator on $L^2(X)$
\[|f\rangle\langle g|: h\mapsto \langle h,g \rangle  f.\]

(i)$\Rightarrow$(ii). 
Assume (i) holds. Let $f=(f_j)_j\subset L^2(X)$ be a Bessel family, and let
$\lambda=(\lambda_j)_j$ be finitely supported.  Define
$T:\ell^2\to L^2(X)$ and $\Lambda:\ell^2\to\ell^2$ by
\[
  Te_j=f_j,
  \qquad
  \Lambda e_j=\lambda_j e_j,
\]
and set
\[
  \Gamma_{f,\lambda}:=T\Lambda T^*
  =\sum_j\lambda_j|f_j\rangle\langle f_j|.
\]
By the definition of the Bessel bound,
$\norm{T}_{\ell^2\to L^2(X)}^2=C_B(f)$.  Since the Bessel family is
normalized,
\[
  \norm{\Gamma_{f,\lambda}}_{\Sch^1(L^2(X))}
  \le \sum_j|\lambda_j|\norm{f_j}_{L^2(X)}^2
  =\norm{\lambda}_{\ell^1},
\]
whereas
\begin{align*}
  \norm{\Gamma_{f,\lambda}}_{\Sch^\infty(L^2(X))}
  &=\norm{T\Lambda T^*}_{L^2(X)\to L^2(X)}\\
  &\le \norm{T}_{\ell^2\to L^2(X)}^2
      \norm{\Lambda}_{\ell^2\to\ell^2}\\
  &\le C_B(f)\norm{\lambda}_{\ell^\infty}.
\end{align*}
Hence, by complex interpolation, for every $1\le\beta<\infty$,
\begin{equation}\label{eq:bessel-schatten-factorization}
  \norm{\Gamma_{f,\lambda}}_{\Sch^\beta(L^2(X))}
  \le C_B(f)^{1-1/\beta}\norm{\lambda}_{\ell^\beta}.
\end{equation}
Moreover, 
we have the 
trace identity
\begin{equation}\label{eq:bessel-duality-trace}
  \Tr\!\left(A_V\Gamma_{f,\lambda}\right)
  =\int_{\R\times X}V(t,x)
    \sum_j\lambda_j|Ef_j(t,x)|^2\,\dd x\dd t.
\end{equation}
For every bounded
$V\in L_t^{p'}(\R;L_x^{q'}(X))$ with support of finite measure,
\eqref{eq:bessel-schatten-factorization}, \eqref{eq:bessel-duality-trace}, 
and Schatten H\"older inequality (see \cite[Theorem~2.8]{SimonTrace}) give
\begin{align*}
 \left| \int_{\R\times X}V(t,x)
    \sum_j\lambda_j|Ef_j(t,x)|^2\,\dd x\dd t\right|
  &\le
  \norm{A_V}_{\Sch^{\beta'}(L^2(X))}
  \norm{\Gamma_{f,\lambda}}_{\Sch^\beta(L^2(X))}\\
  &\le
  C C_B(f)^{1-1/\beta}\norm{\lambda}_{\ell^\beta}
  \norm{V}_{L_t^{p'}(\R;L_x^{q'}(X))}.
\end{align*}
The density of bounded functions with support of finite measure in $L_t^{p'}(\R;L_x^{q'}(X))$ and the mixed-norm duality yield \eqref{eq:intro-bessel-conclusion} for finitely supported
$\lambda$.  The general case follows by approximation in $\ell^\beta$.

(ii)$\Rightarrow$(i). Conversely, assume \textup{(ii)} holds and fix
$V\in L_t^{p'}(\R;L_x^{q'}(X))$.
In this case, $A_V$ is well defined and is a bounded operator on $L^2(X)$.
Let $\Gamma$ be an arbitrary finite-rank operator on $L^2(X)$.  By the
singular value decomposition, 
\[
  \Gamma=\sum_{j=1}^n\lambda_j|f_j\rangle\langle g_j|,
\]
where $n$ is the rank of $\Gamma$, $(f_j)_{j=1}^n$ and $(g_j)_{j=1}^n$ are orthonormal systems and
$\lambda_j>0$.
By the definition of $A_V$,
\[
  \Tr(A_V\Gamma)
  =\int_{\R\times X}V(t,x)
    \sum_{j=1}^n\lambda_j Ef_j(t,x)\overline{Eg_j(t,x)}\,\dd x\dd t.
\]
Therefore, Cauchy--Schwarz, mixed-norm H\"older inequality and \textup{(ii)} give
\begin{align*}
  |\Tr(A_V\Gamma)| & \leq \norm{V}_{L_t^{p'}(\R;L_x^{q'}(X))}
  \norm{\sum_{j=1}^n\lambda_j Ef_j\overline{Eg_j}}_{L_t^p(\R;L_x^q(X))}\\
  &\le C\norm{V}_{L_t^{p'}(\R;L_x^{q'}(X))} \norm{\lambda}_{\ell^\beta}
  =C\norm{V}_{L_t^{p'}(\R;L_x^{q'}(X))}\norm{\Gamma}_{\Sch^\beta(L^2(X))}.
\end{align*}
Since finite-rank operators are dense in $\Sch^\beta$,
by Schatten duality (see \cite[Theorem~3.2]{SimonTrace}),
\begin{equation*}
  \norm{A_V}_{\Sch^{\beta'}(L^2(X))}
  =\sup_{\substack{\norm{\Gamma}_{\Sch^\beta(L^2(X))}\le1}}
    |\Tr(A_V\Gamma)| \le C\norm{V}_{L_t^{p'}(\R;L_x^{q'}(X))}.
\end{equation*}
This proves \textup{(i)} and completes the proof.
\end{proof}

\begin{remark}\label{rem:abstract-closed-subspace}
Lemma~\ref{thm:intro-abstract}
remains valid, if the domain $L^2(X)$ is replaced
by any closed subspace, or even arbitrary Hilbert space.  This will apply to the radial case in Section~\ref{sec:radial}.
\end{remark}

\subsection{From dispersive to Bessel-family Strichartz estimates}
\label{subsec:abstract-dispersive-bessel}

The argument for the following lemma is adapted from \cite[Theorem~9]{FrankSabin} and  \cite[Theorem~1.4]{Hoshiya2025}.

\begin{lemma}
\label{thm:abstract-bessel-dispersive}
Let $X$ be a $\sigma$-finite measure space, and let
$U(t)=e^{-itL}$ be a strongly continuous unitary group on $L^2(X)$ with
self-adjoint operator $L$ as infinitesimal generator.  Let $P:L^2(X)\to L^2(X)$ be a bounded operator
satisfying
\begin{equation*}
  \norm{P}_{L^2(X)\to L^2(X)}\le1,
  \qquad
  PU(t)=U(t)P,
  \qquad t\in\R.
\end{equation*}
Assume that, for some $h>\frac12$ and  $C_0>0$,
\begin{equation}\label{eq:dispersive hypothesis}
  \norm{U(t)PP^*}_{L^1(X)\to L^\infty(X)}
  \le C_0\,|t|^{-h},\qquad t\ne0.
\end{equation}
Let $p,q\ge 1$ satisfy $1/p+h/q=h$, and set $$q_b(h):=\frac{2h+1}{2h-1}.$$  
Suppose that one of the following holds:
\[
\left\{
\begin{aligned}
&1\le \beta\le \beta(q),
&&1\le q<q_b(h),
&&\textup{(a)},
\\
&1\le \beta<p,
&&h\ge1,\quad
q_b(h)\le q<\frac{h}{h-1},
&&\textup{(b)},
\\
&1\le \beta<
\frac{q(2h-1)}{q(2h-1)-1},
&&\frac12<h<1,\quad
q_b(h)\le q<\infty,
&&\textup{(c)},
\\
&\beta=1,
&&\frac12<h<1,\quad q=\infty,
&&\textup{(d)}.
\end{aligned}
\right.
\]
Then for every Bessel family $f=(f_j)_j\subset L^2(X)$ and
every $\lambda\in\ell^\beta$
\begin{equation}\label{eq:bessel-strichartz-estimate}
  \norm{\sum_j\lambda_j|U(t)Pf_j|^2}_{L_t^p(\R;L_x^q(X))}
  \lesssim
  C_0^{\,1-1/q}\,C_B(f)^{1-1/\beta}\,
  \norm{\lambda}_{\ell^\beta}.
\end{equation}
\end{lemma}

\begin{proof}
\emph{Step 1. Reduction.} We may assume $C_0=1$.  For general $C_0>0$, set $\widetilde U(t)=U(C_0^{1/h}t)$.  Then
\[
  \norm{\widetilde U(t)PP^*}_{L^1(X)\to L^\infty(X)}\le |t|^{-h}.
\]
Applying the $C_0=1$ estimate to $\widetilde U$ and changing variables gives
the estimate \eqref{eq:bessel-strichartz-estimate}.

Since $P, P^*$ commute with $U(t)$, for $f(t,x)\in L^2_x(X)$
\[
  (EE^*f)(t,x)=\int_{\R}U(t-s)PP^*f(s,x)\,\dd s, \quad E:=U(t) \circ P.
\]
For $1<p,q<\infty$, Lemma~\ref{thm:intro-abstract} reduces the desired
Bessel-family estimate \eqref{eq:bessel-strichartz-estimate} to
\begin{equation}\label{eq:schatten-normalized}
  \norm{A_V}_{\Sch^{\beta'}(L^2(X))}=
  \norm{E^*M_VE}_{\Sch^{\beta'}(L^2(X))}
  \lesssim
  \norm{V}_{L_t^{p'}(\R;L_x^{q'}(X))}
\end{equation}
for every $V\in L_t^{p'}(\R;L_x^{q'}(X))$,
where $M_V$ is the \emph{multiplication operator} by $V$.
By decomposing the real and imaginary parts of $V$ into their positive and negative parts, it suffices to prove
\eqref{eq:schatten-normalized} for $V\ge0$.

\emph{Step 2.}  We begin with the \emph{smaller subcritical range}
\begin{equation}\label{eq:range-nearB}
  1+\frac1h<q<q_b(h).
\end{equation}
Put $r:=2q'$,
then $2h+1<r<2h+2$.
For $\varepsilon>0$ define the truncated analytic family
\[
  \mathcal T_z^\varepsilon f(t,x)
  :=\int_\R
    \mathbf 1_{\{|t-s|>\varepsilon\}}(t-s)^{-1-z}
    U(t-s)PP^*f(s,x)\,\dd s,
  \qquad -\frac r2\le\Re z\le0,
\]
with kernel denoted by $K_z^\varepsilon(t,x;s,y)$.
Let $M_W$ be a multiplication operator associated with function $W$ on
$\R\times X$.  On the vertical line $z=-r/2+ib$, the dispersive estimate \eqref{eq:dispersive hypothesis} gives
the pointwise bound
\[
  |K_z^\varepsilon(t,x;s,y)|
  \lesssim e^{c|b|}\,1_{\{|t-s|>\varepsilon\}}\,|t-s|^{-1+r/2-h}.
\]
Hence, with $f_W(t):=\norm{W(t,\cdot)}_{L_x^2}^2$,
\begin{equation*}
  \norm{M_W\mathcal T^\varepsilon_{-r/2+ib}M_W}_{\Sch^2}^2
  \lesssim e^{c|b|}
  \iint_{\R^2} 1_{\{|t-s|>\varepsilon\}}|t-s|^{r-2h-2}f_W(t)f_W(s)\,\dd t\dd s.
\end{equation*}
Since $0<2h+2-r<1$, the Hardy--Littlewood--Sobolev
inequality yields, uniformly in $\varepsilon$,
\begin{equation}\label{eq:abstract-Tz-HS-short}
  \norm{M_W\mathcal T^\varepsilon_{-r/2+ib}M_W}_{\Sch^2}
  \lesssim e^{c|b|}
  \norm{W}_{L_t^{4/(r-2h)}L_x^2}^{2}.
\end{equation}
On the other hand, for $z=ib$, since $\norm{PP^*}_{L^2\to L^2}\le1$,
\begin{equation}\label{eq:abstract-Tz-Sinfty-44}
  \norm{M_W\mathcal T^\varepsilon_{ib}M_W}_{\Sch^\infty}
  \lesssim e^{c|b|}\norm{W}_{L_{t,x}^\infty}^{2},
\end{equation}
where we use the uniform boundedness of the truncated Hilbert transform
\[ \norm{H_b^\varepsilon}_{L^2(\R)\to L^2(\R)}
  \lesssim e^{c|b|}, \quad 
  H_b^\varepsilon g(t)
  :=\int_{|t-s|>\varepsilon}\frac{g(s)}{(t-s)^{1-ib}}\,\dd s.
\]
Applying Stein's interpolation to \eqref{eq:abstract-Tz-HS-short} and
\eqref{eq:abstract-Tz-Sinfty-44} gives
\begin{equation}\label{eq:abstract-weighted-TTstar-schatten-short}
  \norm{M_W\mathcal T^\varepsilon_{-1}M_W}_{\Sch^r}
  \lesssim
  \norm{W}_{L_t^{2r/(r-2h)}L_x^r}^{2},
  \qquad 2h+1<r<2h+2.
\end{equation}
Now take $W=V^{1/2}$ with $V\ge0$ and let $\varepsilon \to 0$. Since the singular values of
$(M_WE)(M_WE)^*$ and $(M_WE)^*(M_WE)=E^*M_VE$ coincide,
\eqref{eq:abstract-weighted-TTstar-schatten-short} gives \eqref{eq:schatten-normalized} and \eqref{eq:bessel-strichartz-estimate} with $\beta=\beta(q)$ for the range \eqref{eq:range-nearB}.

\emph{Step 3. Interpolation.}
Then, for $1\leq p,q \leq \infty$ satisfying $1/p+h/q=h$, with $(h,p,q)\ne(1,1,\infty)$,
the dispersive estimate
\eqref{eq:dispersive hypothesis} and the standard Keel--Tao argument (see \cite[Theorem~1.2]{KeelTao}) give the classical single-function Strichartz estimate
\begin{equation*}
  \norm{U(t)Pf}_{L_t^{2p}(\R;L_x^{2q}(X))}
  \lesssim \norm{f}_{L^2(X)}.
\end{equation*}
By the triangle inequality we have the  $\ell^1$ bound 
\begin{equation*}
  \norm{\sum_j\lambda_j|U(t)Pf_j|^2}_{L_t^p(\R;L_x^q(X))}
  \le \sum_j|\lambda_j|
      \norm{U(t)Pf_j}_{L_t^{2p}(\R;L_x^{2q}(X))}^2
  \lesssim \norm{\lambda}_{\ell^1}.
\end{equation*}
Complex interpolation with Step 2 gives the desired estimate \eqref{eq:bessel-strichartz-estimate} for all cases \textup{(a)}--\textup{(d)} with $C_0=1$.
More precisely, 
we consider \eqref{eq:bessel-strichartz-estimate} at an arbitrary fixed point $B_*$ such that $1+1/h<q<q_b(h)$. 
Then interpolation with the $\ell^1$ estimate at point $(p,q)=(\infty,1)$ gives the case \textup{(a)};
interpolation with the $\ell^1$ estimate at arbitrary point $(p,q)$ where $q_b(h)\le q<h/(h-1)$ gives the case \textup{(b)} for $h\ge 1$;
interpolation with the $\ell^1$ estimate at point $(p,q)=(1/h, \infty)$ gives the cases \textup{(c)} and \textup{(d)} for $1/2<h<1$.
\end{proof}

\section{KP and ZK models}
\label{sec:KP}
The general almost-orthogonal Schatten duality and dispersive-transference principles can be applied to a wide class of dispersive equations. Here we consider two special models. 

\subsection{KP estimates}

We recall the KP scaling.  Set 
\[
  f_\mu(x,y)=\mu^{3/2}f(\mu x,\mu^2y),\quad \mu>0.
\]
Then a change of variables gives $\norm{f_\mu }_{L^2(\R^2)}=\norm{f}_{L^2(\R^2)}$ and
\[
  \norm{U_{\KP}(t)f_\mu}_{L_t^{2p}(\R;L_{x,y}^{2q}(\R^2))}
  =\mu^{\frac32-\frac{3}{2q}-\frac{3}{2p}}
  \norm{U_{\KP}(t)f}_{L_t^{2p}(\R;L_{x,y}^{2q}(\R^2))}.
\]
Thus a global Strichartz estimate 
can be scale invariant only if
\[
  \frac1p+\frac1q=1.
\]
The kernel of the KP propagator $U_{\KP}(t)$ is
\begin{equation*}
  K_{\KP}(t,x,y)
  =(2\pi)^{-2}\int_{\R^2}e^{i(x\xi+y\eta+t\Phi_{\KP}(\xi,\eta))}
  \,\dd\xi\,\dd\eta,
\end{equation*}
size bound of which will give the following global dispersive estimate; see \cite[Lemma~4]{MSTBackgroundKPI} and the dispersive bound
in the proof of \cite[Theorem~3.1]{HadacKP}.   Theorem~\ref{thm:intro-KP} is then a direct consequence of Lemma~\ref{lem:global-KP-dispersive} and Lemma~\ref{thm:abstract-bessel-dispersive}~{(a) and (b)} with $h=1$,
$P=I$ and $C_0\sim 1$.
\begin{lemma}
\label{lem:global-KP-dispersive}
For every $t\ne0$ and every $f\in L^1(\R^2)$,
\begin{equation*}
  \norm{U_{\KP}(t)f}_{L^\infty(\R^2)}
  \lesssim |t|^{-1}\norm{f}_{L^1(\R^2)}.
\end{equation*}
\end{lemma}

\subsection{ZK estimates}

A similar ZK scaling 
\[
  f_\mu (z)=\mu^{d/2}f(\mu z), \quad \mu>0,
\]
implies that for valid global ZK Strichartz estimate
\[
  \frac3p+\frac d q= d.
\]
For localized estimate in the full admissible region, the frequency power $N^{\sigma_{\ZK}(p,q)}$ appears.
The corresponding
ZK kernel is
\[
  K_{\ZK}(t,z)
  =(2\pi)^{-d}\int_{\R^d}e^{i(z\cdot\zeta+t\xi(\xi^2+|\eta|^2))}\,\dd\zeta,
\]
bound of which will give the following global dispersive estimates.

\begin{lemma}
\label{lem:ZK-dispersive}
For every $t\ne0$ and every $f\in L^1(\R^d)$,
\begin{equation}
  \begin{cases}
 \norm{U_{\ZK}(t)f}_{L^\infty(\R^2)}
  \lesssim |t|^{-2/3}\norm{f}_{L^1(\R^2)}, &  d=2\,,\label{eq:zk-2d-global-dispersive}
  \\[1.2ex]
   \norm{U_{\ZK}(t)P_Nf}_{L^\infty(\R^d)}
  \lesssim N^{d-3}\,|t|^{-1}\norm{f}_{L^1(\R^d)}, & d\ge 3\,.  \end{cases}
\end{equation}
\end{lemma}

\begin{proof}
For $d=2$, the dispersive estimate \eqref{eq:zk-2d-global-dispersive} is given by
\cite[(2.1)--(2.2)]{LinaresPastorDrumondSilva}.

For \(d\ge3\),  by \cite[Lemma~3.2]{LinaresRamosZK}, for every smooth radial
cutoff \(\psi\) supported in a fixed annulus,
\begin{equation}\label{eq:zk-unit-annulus}
  \sup_{z\in\R^d}
  \left|
  \int_{\R^d}
  \psi(\tilde{\zeta})e^{i(t\tilde{\zeta_1}|\tilde{\zeta}|^2+z\cdot\tilde{\zeta})}\,\dd\tilde{\zeta}
  \right|
  \lesssim |t|^{-1}.
\end{equation}
Let $K_N(t,z)$ be the kernel of $U_{\ZK}(t)P_N$.
After the change of variables \(\zeta=N\tilde{\zeta}\),
\[
  K_N(t,z)
  =(2\pi)^{-d}N^d\int_{\R^d}
  e^{i(Nz\cdot\tilde{\zeta}+tN^3\tilde{\zeta_1}|\tilde{\zeta}|^2)}
  \psi(\tilde{\zeta})\,\dd\tilde{\zeta}.
\]
Applying \eqref{eq:zk-unit-annulus} with \(t\rightarrow tN^3\) gives
\[
  |K_N(t,z)|\lesssim N^d|tN^3|^{-1}
  =N^{d-3}|t|^{-1}.
\]
Hence the second inequality in \eqref{eq:zk-2d-global-dispersive}  holds.
The same argument applies to the cutoff $\psi^2$,
while the multiplier $P_N$ replaced by $P_N^2$. 
\end{proof}

\begin{proof}[Proof of Theorem~\ref{thm:intro-ZK}]
Case $d=2$ is a consequence of the first inequality in Lemma~\ref{lem:ZK-dispersive} and Lemma~\ref{thm:abstract-bessel-dispersive} (a), (c) and (d) with $h=2/3,\, P=I$ and $C_0\sim 1$. Case $d\ge 3$, similar to KP case, is a consequence of the second inequality in Lemma~\ref{lem:ZK-dispersive} and Lemma~\ref{thm:abstract-bessel-dispersive} (a) and (b) with $h=1$, $P=P_N$ and $C_0\sim N^{d-3}$. 
\end{proof}

\subsection{Necessary conditions}\label{sec:KP-necessity}

We now determine the necessary conditions on the system exponent $\beta$ for KP and ZK models. 
We first give a general lemma, extending the argument of \cite[Proposition~9]{BLN}.

\begin{lemma}
\label{lem:BLN-local-necessary}
Let $d\ge1$. Let $\Omega\subset\R^d$ be a bounded open set and
$\phi\in C^1(\Omega)$.  Define
\[E_\phi h(t,x)
  :=\int_\Omega e^{i(x\cdot\zeta+t\phi(\zeta))}h(\zeta)\,\dd\zeta.\]
 Suppose that, for some $1\le p\le\infty$, $1\le q<\infty$ and
$\beta\ge1$,  \eqref{eq:intro-bessel-conclusion} holds with $E=E_\phi$. 
Then
\begin{equation}\label{eq:BLN-local-summability-bound}
  \frac1\beta\ge \frac1{dp}+\frac1q.
\end{equation}
If, in addition, $p<\infty$ and
$\nabla\phi(\zeta_0)\ne0$ at some point $\zeta_0\in\Omega$, then
\begin{equation}\label{eq:BLN-beta-p-bound}
  \beta\le p.
\end{equation}
\end{lemma}

\begin{proof}
\emph{Step 1: The condition \eqref{eq:BLN-local-summability-bound}.}
Choose a cube $Q$ of side length $\ell>0$ with
$\overline{Q}\subset\Omega$, and set
$A=\sup_{\zeta\in Q}|\nabla\phi(\zeta)|<\infty$.  For an integer $K\ge1$, partition $Q$ into $K^d$ congruent cubes
$(Q_\nu)_\nu$ with centers $(\zeta_\nu)_\nu$ and side length $\ell/K$.  Define $g_\nu=|Q_\nu|^{-1/2}\mathbf1_{Q_\nu}$.  
Then $(g_\nu)_\nu$ is an orthonormal system in $L^2(\Omega)$.   
For $\zeta\in Q_\nu$, convexity of $Q_\nu$ and the mean value theorem give
\[
  |\phi(\zeta)-\phi(\zeta_\nu)|\le A|\zeta-\zeta_\nu|,
  \qquad
  |\zeta-\zeta_\nu|\le \frac{\sqrt d\,\ell}{2K}.
\]

Consequently, one may choose a constant $c>0$, depending only on $Q$ and
$\phi$, such that, whenever $|x|\le cK$ and $|t|\le cK$,
\[
  \left|
    x\cdot(\zeta-\zeta_\nu)
    +t\bigl(\phi(\zeta)-\phi(\zeta_\nu)\bigr)
  \right|
  \le \frac{\pi}{3},
  \qquad \zeta\in Q_\nu.
\]
Hence for such $t$ and $x$,
\begin{align*}
\sum_\nu |E_\phi g_\nu(t,x)|^2
&=
\sum_\nu |Q_\nu|^{-1}
    \left|
      \int_{Q_\nu}
      e^{i\{x\cdot(\zeta-\zeta_\nu)
      +t(\phi(\zeta)-\phi(\zeta_\nu))\}}\,\dd\zeta
    \right|^2\\
&\ge \frac14\sum_\nu |Q_\nu|
\gtrsim 1.
\end{align*}
The last estimate holds since there are $K^d$ sub-cubes and $|Q_\nu| \sim K^{-d}$.
It follows that
\[
  \left\|\sum_\nu |E_\phi g_\nu|^2\right\|_{L_t^p(\R;L_x^q(\R^d))}
  \gtrsim K^{1/p+d/q}.
\]
Applying \eqref{eq:intro-bessel-conclusion} with
$\lambda_\nu=1$ for all $\nu$ yields
\[
  K^{1/p+d/q}
  \lesssim \norm{(1,\ldots,1)}_{\ell^\beta}
  =K^{d/\beta},
\]
which shows \eqref{eq:BLN-local-summability-bound} after letting $K\to\infty$.

\emph{Step 2: The condition \eqref{eq:BLN-beta-p-bound}.}
Assume
that $p<\infty$ and that $\nabla\phi(\zeta_0)\ne0$ for some
$\zeta_0\in\Omega$.  After permuting the coordinates, we may suppose that
$\partial_{\zeta_1}\phi(\zeta_0)\ne0$.  By the inverse function theorem,
after restricting to a smaller open subset $\Omega_0$ with 
$\overline{\Omega_0}\subset\Omega$, the map
\[
  \zeta\longmapsto \bigl(\phi(\zeta),\zeta_2,\ldots,\zeta_d\bigr)
\]
is a $C^1$ diffeomorphism.  Thus there are open sets
$I_0\subset\R$ and $\Theta\subset\R^{d-1}$ and a $C^1$ diffeomorphism
$\Psi:I_0\times \Theta\to\Omega_0$ such that
\[
  \phi(\Psi(l,\vartheta))=l,
  \qquad
  \zeta=\Psi(l,\vartheta),
  \qquad
  \dd\zeta=J(l,\vartheta)\,\dd l\dd\vartheta,
  \qquad J(l,\vartheta)>0.
\]
Choose a compact interval $I=[a,a+2\pi/T]\subset I_0$ for some $T>0$ and a function
$\varphi\in C_c^\infty(\Theta)$ with $\norm{\varphi}_{L^2(\Theta)}=1$.  Define $h\in L^2(\Omega)$ by
\[
  h(\Psi(l,\vartheta))
  =\left(\frac{T}{2\pi}\right)^{1/2}
  \mathbf1_I(l)\varphi(\vartheta)J(l,\vartheta)^{-1/2},
\]
and set $h=0$ outside $\Omega_0$.  Then
$\norm{h}_{L^2(\Omega)}=1$.  For $j\in\Z$, set
\[
  g_j(\zeta)=e^{-ijT\phi(\zeta)}h(\zeta).
\]
A change of variables gives
\[
  \ip{g_j}{g_k}_{L^2(\Omega)}
  =\frac{T}{2\pi}\int_a^{a+2\pi/T}e^{-i(j-k)Ts}\,\dd s
  =\delta_{jk}.
\]
So $(g_j)_{j\in\Z}$ is an orthonormal system.  Moreover,
\[
  E_\phi g_j(t,x)=E_\phi h(t-jT,x).
\]
Since $h$ is supported in the bounded set $\Omega_0$, the Hausdorff--Young inequality and dominated
convergence theorem show that the map $t\mapsto E_\phi h(t,\cdot)$ is
continuous in $L_x^{2q}(\R^d)$.  Also, $E_\phi h(0,\cdot)$ is not
identically zero.  Hence there exist
$\delta,c_0>0$ such that
\[
  \norm{E_\phi h(t,x)}_{L_x^{2q}(\R^d)}^2\ge c_0,
  \qquad |t|\le\delta.
\]
We may assume that $0<\delta<T/2$, so the
intervals $[jT-\delta,jT+\delta]$ are pairwise disjoint.  Therefore, for
any positive integer $M$,
\begin{align*}
  \left\|\sum_{j=1}^M|E_\phi g_j|^2\right\|_{L_t^p(\R;L_x^q(\R^d))}^p
  &\ge
  \sum_{j=1}^M\int_{jT-\delta}^{jT+\delta}
  \norm{E_\phi g_j(t,x)}_{L_x^{2q}(\R^d)}^{2p}\,\dd t
  \\
  &=
  M\int_{-\delta}^{\delta}
  \norm{E_\phi h(t,x)}_{L_x^{2q}(\R^d)}^{2p}\,\dd t
\  \gtrsim M.
\end{align*}
Applying \eqref{eq:intro-bessel-conclusion} to the orthonormal system $(g_j)^M_{j=1}$ with
$\lambda_j=1$ yields
\[
  M^{1/p}\lesssim M^{1/\beta},
\]
which gives
\eqref{eq:BLN-beta-p-bound} after letting $M\to\infty$.
\end{proof}

\begin{proposition}
\label{prop:KP-ZK-necessary}
For KP/ZK models, if the Bessel-family Strichartz estimate holds, then 
\begin{equation*}
\emph{(KP)} \quad \beta\le
  \begin{cases}
    \beta(q),&1\le q<3,\\[1.2ex]
    p,&3\le q<\infty.
  \end{cases}  
  \qquad
\emph{(2D-ZK)} \quad \beta\le
  \begin{cases}
    \dfrac{3q}{q+2},&1\le q<4,\\[1.2ex]
    p,&4\le q<\infty.
  \end{cases}
\end{equation*}

\begin{equation}\label{eq:ZK-combined-necessary}
\emph{($3^{\scriptscriptstyle +}$D-ZK)} \quad \beta\le
  \begin{cases}
    \dfrac{dq}{q+d-1},
      &1\le q< \dfrac{2d-1}{d-1},\\[1.5ex]
    p,
      &\dfrac{2d-1}{d-1}\le q<\infty;
  \end{cases} \quad d\ge 3.
\end{equation}
\end{proposition}

\begin{proof}
Fix a bounded open  set $\Omega\subset\R^d$ in frequency, to be chosen later,
and let $(h_j)_j$ be an arbitrary finite orthonormal system in
$L^2(\Omega)$.  Denote by $\widetilde h_j$ the extension of $h_j$ by zero
to $\R^d$, and define $f_j=\widecheck{\widetilde{h_j}} \in L^2(\R^d)$. 
Because the Fourier transform is
unitary on $L^2$, the family $(f_j)_j$ is orthonormal in
$L^2(\R^d)$.
For the KP case, choose $\Omega$ to be a sufficiently small neighborhood of
$(1,0)$ contained in $\{(\xi,\eta):\xi\ne0\}$ and take
$\phi=\Phi_{\KP}$.  For the 2D-ZK case, choose $\Omega$ to be a sufficiently small
neighborhood of $(1,0)$ and set $\phi=\Phi_{\ZK}$.  For the
$3^{\scriptscriptstyle +}$D-ZK case, choose $\Omega$ to be a small neighborhood of
$(1,0,\ldots,0)$ contained in the region on which the multiplier
$P_1$ is identically one and set $\phi=\Phi_{\ZK}$. In all three
cases the phase $\phi \in C^1(\Omega)$ and has nonvanishing gradient at the
chosen center, since
\[
  \partial_\xi\Phi_{\KP}(1,0)=3,
  \qquad
  \partial_\xi\Phi_{\ZK}(1,0,\ldots,0)=3.
\]
In particular, in all cases, \(U(t)\) denotes the operator appearing in the corresponding Strichartz estimate, and
\[
  U(t)f_j(x)
  =(2\pi)^{-d/2}\int_\Omega
    e^{i(x\cdot\zeta+t\phi(\zeta))}h_j(\zeta)\,\dd\zeta
  =(2\pi)^{-d/2}E_\phi h_j(t,x).
\]
Therefore, the necessary conditions \eqref{eq:ZK-combined-necessary} follow directly from Lemma~\ref{lem:BLN-local-necessary}, 
together with \eqref{eq:intro-bessel-conclusion} and the scaling relation on the sharp line in each case.
\end{proof}

\section{Improvement for the radial Schr\"odinger equation}\label{sec:radial}

We finally turn to the Schr\"odinger
equation \eqref{eq:Scheq}, 
restricted to radial initial
data. 
\cite[Theorem~8]{FrankSabin} gives orthonormal Strichartz estimates in the common (nonradial) range
\[
 \frac{2}{p}+\frac{d}{q}=d,
 \qquad
 1\le q<\frac{d+1}{d-1}, \quad d\ge 1.
\]
For $d\ge 2$, radial symmetry in the classical single-function case gains a larger range by ruling out angular Knapp-type
concentration. 
However, the Schr\"odinger propagator satisfies the same (sharp)
dispersive estimate for radial or nonradial data. Hence the dispersive-to-Strichartz principle in Lemma~\ref{thm:abstract-bessel-dispersive} used
for KP/ZK models, cannot detect the improvement caused by radial symmetry.
Accordingly, our argument is formulated through a robust Schatten bound of the 
frequency-localized $T^*T$ operator.
For $N>0$,  
set the extension operator
\begin{equation}\label{eq:EN}
  E_Nf(t,x):=U_S(t)P_Nf(x),
  \qquad
  f\in L^2_{\rad}(\R^d).
\end{equation}
Note that both the Schr\"odinger
propagator $U_S(t)$ and the smooth multiplier $P_N$ preserve the radiality. 
For a space--time function $V$, define the $T^*T$ operator
\begin{equation}\label{eq:AVN}
  A_{V,N}^{\rad}:=E_N^*M_VE_N,
\end{equation}
where $M_V$ is the multiplication operator by $V$. 
Then the main
step for desired Strichartz estimates is to establish suitable Schatten bounds for $A_{V,N}^{\rad}$, in view of Lemma~\ref{thm:intro-abstract}. 
Besides, we can restrict $V$ to be radial in the spatial variable, and then $A_{V,N}^{\rad}$ also preserves radiality.  If
$f,g\in L^2_{\rad}(\R^d)$, then
$U_S(t)f\,\overline{U_S(t)g}$ is radial in $x$.  For $V\in L_t^{p'}(\R;L_x^{q'}(\R^d))$, define its spherical average as a radial function by
\[
 V_{\mathrm{avg}}(t,r)
 =
 \frac{1}{|\mathbb S^{d-1}|}
 \int_{\mathbb S^{d-1}}V(t,r\omega)\,\dd\omega.
\]
Then by rotation it is easy to check
\[
 \int_{\R\times\R^d}V(t,x)U_S(t)f(x)
 \overline{U_S(t)g(x)}\,\dd x\dd t
 =
 \int_{\R\times\R^d}V_{\mathrm{avg}}(t,|x|)U_S(t)f(x)
 \overline{U_S(t)g(x)}\,\dd x\dd t.
\]
Moreover, Jensen's inequality yields
\[
 \norm{V_{\mathrm{avg}}}_{
 L_t^{p'}\bigl(\R;
 L_r^{q'}((0,\infty),r^{d-1}\dd r)\bigr)}
 \lesssim
 \norm{V}_{L_t^{p'}(\R;L_x^{q'}(\R^d))}.
\]

We will first prove a new key diagonal endpoint Schatten estimate, then obtain the remaining estimates by interpolation, and finally establish some
necessary conditions for the system exponents $\beta$ with fixed $p,q$.

\subsection{The diagonal radial Schatten estimate and preliminary}
\label{subsec:radial-diagonal-endpoint}

Set 
\[  \beta_*:=\beta(q_*)=\frac{2d+1}{2d},\]
with $\beta(\cdot)$ from \eqref{eq:betaq} and $q_*$ from \eqref{eq:intro-radial-parameters}.
Then 
$\beta_*'=2d+1=2q_*'$.  
\begin{theorem}
\label{thm:radial-diagonal-schatten}
Let $d\ge2$. For every $N>0$ and $V\in L_{t,x}^{q_*'}(\R\times\R^d)$,
for $A_{V,N}^{\rad}$ defined in \eqref{eq:AVN},
\begin{equation}\label{eq:radial-diagonal-schatten}
  \norm{A_{V,N}^{\rad}}_{\Sch^{\beta_*'}(L^2_{\rad}(\R^d))}
  \lesssim
  N^{-\frac{2(d-1)}{2d+1}}
  \norm{V}_{L_{t,x}^{q_*'}(\R\times\R^d)}.
\end{equation}
\end{theorem}

We first establish several auxiliary lemmas. The first one is a general uniform one-dimensional Schatten estimate.

\begin{lemma}
\label{lem:radial-time-schatten}
Let $I\subset(0,\infty)$ be a fixed compact interval, and let
$(h_\alpha)_\alpha\subset C^1(I)$ satisfy
\[
  \sup_\alpha
  \bigl(\norm{h_\alpha}_{L^\infty(I)}
  +\norm{h_\alpha'}_{L^\infty(I)}\bigr)<\infty.
\]
For $z,t\in\R$, define
\[
  \kappa_{z,\alpha}(t)
  =\int_I e^{iz\rho-it\rho^2}h_\alpha(\rho)\,\dd\rho, \qquad \mathcal{C}_{z,\alpha}f=\kappa_{z,\alpha}*f, \quad \forall\,f\in L^2(\R).
\]
Then, uniformly in $\alpha$, for every $\gamma\ge2$, $z\in \R$, and $a,b\in L^{\gamma}(\R)$,
\begin{equation}\label{eq:radial-time-schatten}
  \norm{M_a\mathcal{C}_{z,\alpha}M_b}_{\Sch^{\gamma}(L^2(\R))}
  \lesssim_{\gamma,I}
  \langle z\rangle^{-1/\gamma}
  \norm{a}_{L^{\gamma}(\R)}\norm{b}_{L^{\gamma}(\R)}.
\end{equation}
\end{lemma}

\begin{proof}
We first estimate the $\Sch^\infty$ norm.  
Taking the Fourier transform, 
\[
  \widehat{\kappa_{z,\alpha}}(\sigma)
  =(2\pi)^{1/2}\mathbf 1_{-I^2}(\sigma)
  e^{iz\sqrt{-\sigma}}
  \frac{h_\alpha(\sqrt{-\sigma})}{2\sqrt{-\sigma}},
\]
where $-I^2:=\{-\rho^2:\rho\in I\}$.
Since $I$ is a compact subinterval of $(0,\infty)$ and
$(h_\alpha)_\alpha$ is uniformly bounded on $I$, 
$\widehat{\kappa_{z,\alpha}}$ is also uniformly bounded.
Consequently, by Plancherel's theorem,
\begin{align}
  \norm{M_a\mathcal{C}_{z,\alpha}M_b}_{\Sch^\infty}
  & \le \norm{M_a}_{L^2(\R)\to L^2(\R)}\norm{\mathcal{C}_{z,\alpha}}_{L^2(\R)\to L^2(\R)}\norm{M_b}_{L^2(\R)\to L^2(\R)} \notag
  \\ & \lesssim \norm{a}_{L^\infty(\R)}\norm{b}_{L^\infty(\R)}
    \norm{\widehat{\kappa_{z,\alpha}}}_{L^\infty(\R)} \notag
  \\ & \lesssim \norm{a}_{L^\infty(\R)}\norm{b}_{L^\infty(\R)}.
  \label{eq:radial-time-operator-endpoint}
\end{align}

We next estimate the $\Sch^2$ norm.
By change of variables
$\mu=\rho^2$, 
\[
  \kappa_{z,\alpha}(t)
  =\int_{I^2} e^{i\Phi_{z,t}(\mu)}
  \widetilde h_\alpha(\mu)\,\dd\mu,
\]
where
\[
  \widetilde h_\alpha(\mu)
  :=\frac{h_\alpha(\sqrt\mu)}{2\sqrt\mu},
  \qquad
  \Phi_{z,t}(\mu):=z\sqrt\mu-t\mu,
  \qquad \mu\in I^2.
\]
In particular,
$(\widetilde h_\alpha)_\alpha$ is
supported in the fixed compact interval $I^2$ and is uniformly bounded in $C^1(I^2)$.  
If $|z|\ge1$, the phase $\Phi_{z,t}(\mu)$ satisfies
\[
  |\Phi''_{z,t}(\mu)|=\frac{|z|}{4\mu^{3/2}}\gtrsim |z|
\]
uniformly on $I^2$.  Van der Corput's lemma (see, e.g.,
\cite[Chapter~VIII]{SteinHA}) therefore gives
\[
  \sup_{t\in\R}|\kappa_{z,\alpha}(t)|\lesssim |z|^{-1/2}, \qquad |z|\ge1.
\]
For $|z|<1$, 
\[
  |\kappa_{z,\alpha}(t)|
  \le \int_I|h_\alpha(\rho)|\,\dd\rho
  \le |I|\norm{h_\alpha}_{L^\infty(I)}
  \lesssim1.
\]
Hence using the Japanese bracket
\begin{equation*}
  \norm{\kappa_{z,\alpha}}_{L^\infty(\R)}
  \lesssim \langle z\rangle^{-1/2}.
\end{equation*}
The integral kernel of $M_a\mathcal{C}_{z,\alpha}M_b$ is
$a(t)\kappa_{z,\alpha}(t-s)b(s)$.  It follows that
\begin{align}
  \norm{M_a\mathcal{C}_{z,\alpha}M_b}_{\Sch^2}^2
  &=\iint_{\R^2}|a(t)|^2
    |\kappa_{z,\alpha}(t-s)|^2|b(s)|^2
    \,\dd t\dd s \notag\\
  &\lesssim \langle z\rangle^{-1}
    \norm{a}_{L^2(\R)}^2\norm{b}_{L^2(\R)}^2.
  \label{eq:radial-time-HS-endpoint}
\end{align}
Bilinear complex interpolation between
\eqref{eq:radial-time-operator-endpoint} and
\eqref{eq:radial-time-HS-endpoint} yields \eqref{eq:radial-time-schatten}.  
\end{proof}

We also record the following vector-valued multi-product trace formula. We recall the vector-valued function space $L^p(X;H)$ for some measure space $X$ and  Banach space $H$, consisting of strongly measurable functions $f:X\rightarrow H$ with
\[\|f\|_{L^p(X;H)}^p=\int_X \|f(x)\|_H^p\,\dd x<\infty.\]

\begin{lemma}
\label{lem:operator-valued-cyclic-trace}
Let $(X,\mu)$ be a $\sigma$-finite measure space such that $L^2(X)$ is separable, 
let $H$ be a separable Hilbert space, and let $\ell\ge2$ be an
integer.  Suppose that  
\[
  D_j\in L^2\bigl(X\times X;\Sch^2(H)\bigr),  \qquad j=1,\ldots,\ell.
\]
Let $B_j$ be the Hilbert--Schmidt integral operator on $L^2(X;H)$ with
kernel $D_j$, namely
\begin{equation}\label{eq:Bj}
  (B_jf)(x)=\int_XD_j(x,y)f(y)\,\dd\mu(y).
\end{equation}
Then the composition operator
$B=\prod_{j=1}^\ell B_j$ 
belongs to
$\Sch^1(L^2(X;H))$ with trace
\begin{equation}\label{eq:operator-valued-cyclic-trace}
  \Tr_{L^2(X;H)}(B)
  =\int_{X^\ell}
  \Tr_H\left(
    \prod_{j=1}^{\ell}D_j(x_j,x_{j+1})
  \right)
  \prod_{j=1}^\ell\dd\mu(x_j),
\end{equation}
where we set $x_{\ell+1}=x_1$.
The integral on the r.h.s. is absolutely convergent.
\end{lemma}

\begin{proof}
The Schatten H\"older inequality gives
\begin{align*}
  \left|\Tr_H\left(
    \prod_{j=1}^{\ell}D_j(x_j,x_{j+1})
  \right)\right|
  &\le \norm{D_1(x_1,x_2)D_2(x_2,x_3)}_{\Sch^1(H)} \prod_{j=3}^\ell \norm{D_j(x_j,x_{j+1})}_{\Sch^\infty(H)} \nonumber\\
  &\le\prod_{j=1}^\ell \norm{D_j(x_j,x_{j+1})}_{\Sch^2(H)}.
\end{align*}
Hence by Finner's generalized H\"older inequality
\cite[Theorem~2.1]{Finner1992} 
\begin{equation}\label{eq:operator-valued-cyclic-integral-bound}
  \int_{X^\ell}
  \left|\Tr_H\left(
    \prod_{j=1}^{\ell}D_j(x_j,x_{j+1})
  \right)\right|
  \prod_{j=1}^{\ell}\dd\mu(x_j)
  \le
  \prod_{j=1}^{\ell}
  \norm{D_j}_{L^2(X\times X;\Sch^2(H))}.
\end{equation}
Thus the cyclic integral is absolutely convergent.
On the other hand,
by the vector-valued Hilbert--Schmidt kernel identification
\cite[Lemma~B.3]{HundertmarkKunstmannRiedVugalter2023},
\[
  \norm{B_j}_{\Sch^2(L^2(X;H))}
  =\norm{D_j}_{L^2(X\times X;\Sch^2(H))},
  \qquad 
  j=1,\ldots,\ell.
\]
Therefore $B\in\Sch^1(L^2(X;H))$ as 
\begin{equation}\label{eq:operator-valued-cyclic-trace-bound}
  \left|\Tr_{L^2(X;H)}(B
  )\right|
  \le \norm{B_1B_2}_{\Sch^1(L^2(X;H))}
       \prod_{j=3}^{\ell}\norm{B_j}_{\Sch^\infty(L^2(X;H))} 
  \le \prod_{j=1}^{\ell}
       \norm{D_j}_{L^2(X\times X;\Sch^2(H))}.
\end{equation}
From \eqref{eq:operator-valued-cyclic-integral-bound} and
\eqref{eq:operator-valued-cyclic-trace-bound}, we see both sides of
\eqref{eq:operator-valued-cyclic-trace} define continuous $\ell$-linear
functionals on the product space $L^2(X\times X;\Sch^2(H))^\ell$.

The proof of \cite[Lemma~B.3]{HundertmarkKunstmannRiedVugalter2023} also shows that the subclass of finite linear
combinations of kernels like
\[
  D(x,y)=u(x)v(y)A,
  \qquad u,v\in L^2(X),  \quad A\in \mathfrak S^2(H)~ \text{finite rank}
\]
are dense in
$L^2(X\times X;\Sch^2(H))$.  
It therefore suffices to prove for 
\[
  D_j(x,y)=u_j(x)v_j(y)A_j,
  \qquad j=1,\ldots,\ell,
\]
with each $A_j\in \mathfrak S^2(H)$ finite rank and $u_j, v_j\in L^2(X)$.  
By the definition of $B_j$, 
\[
  (B f)(x)
  =
  u_1(x)
  \prod_{j=1}^\ell A_j
    \prod_{j=1}^{\ell-1}
    \int_X v_j(y)u_{j+1}(y)\,\dd\mu(y)
  \int_X v_\ell(y)f(y)\,\dd\mu(y).
\]
 Since $A_j$ has finite rank,
a direct trace computation with convention $u_{\ell+1}=u_1$ gives
\[
  \Tr_{L^2(X;H)}(B)
  =
  \Tr_H\left(\prod_{j=1}^\ell A_j\right)
  \prod_{j=1}^{\ell}
  \int_X v_j(x)u_{j+1}(x)\,\dd\mu(x).
\]
On the other hand, since
$v_ju_{j+1}\in L^1(X)$ by Cauchy--Schwarz, Fubini's theorem gives
\begin{align*}
  \int_{X^\ell}
  \Tr_H\!\left(
    \prod_{j=1}^{\ell}D_j(x_j,x_{j+1})
  \right)
  \prod_{j=1}^{\ell}\dd\mu(x_j)
  =
  \Tr_H\left(\prod_{j=1}^\ell A_j\right)
  \prod_{j=1}^{\ell}
  \int_X v_j(x)u_{j+1}(x)\,\dd\mu(x).
\end{align*}
Thus \eqref{eq:operator-valued-cyclic-trace} holds for above
subclass of kernels, and hence for arbitrary $D_1,\ldots,D_\ell \in L^2(X\times X;\Sch^2(H))$ by density.
\end{proof}

Finally, we record the following weighted fractional integral estimate in a general multilinear cyclic form.

\begin{lemma}
\label{lem:radial-cyclic-weighted}
Let $\ell\ge3$ be an integer and let
$\ell/(\ell-1)<\tau \le \ell$.
Set
\[
  \theta=\theta(\ell,\tau):=1-\frac1\ell-\frac1\tau.
\]
Then, for measurable functions $(F_j)_{j=1}^\ell$ 
such that
$r^\theta F_j(r)\in L^\tau((0,\infty))$, 
\begin{equation}\label{eq:radial-cyclic-weighted-general}
  \left|\int_{(0,\infty)^\ell}
  \prod_{j=1}^{\ell}F_j(r_j)\,
  \prod_{j=1}^{\ell}|r_j-r_{j+1}|^{-1/\ell}
  \,\dd r_j\right|
  \lesssim
  \prod_{j=1}^{\ell}
  \norm{r^\theta F_j(r)}_{L^\tau((0,\infty))},
\end{equation}
with convention $r_{\ell+1}=r_1$.
\end{lemma}

\begin{proof}
It suffices to prove the corresponding estimate with $|F_j|$ in place of
$F_j$, hence we
may assume that $F_j\ge0$.
For arbitrary nonnegative $G_1,\ldots,G_\ell\in L^\tau(\R)$ and $x\in\R$, define
\[
\begin{aligned}
  T(G_1,\ldots,G_\ell)(x)
  :=\int_{\R^\ell}
  \prod_{j=1}^{\ell}G_j(r_j)
  \prod_{j=1}^{\ell}|r_j-r_{j+1}|^{-1/\ell}
  \prod_{j=1}^{\ell}|r_j-x|^{-\theta}
  \,\dd r_j.
\end{aligned}
\]
We first apply the $L^\infty$ estimate in
\cite[Theorem~1.2]{ShiWuYan2019} to
$G_j\in C_c^\infty(\R)$ and then extend to general $G_j\in L^\tau(\R)$, to obtain
\begin{equation}\label{eq:radial-cyclic-correlation-bound}
  \norm{T(G_1,\ldots,G_\ell)}_{L^\infty(\R)}
  \lesssim_{\ell,\tau}
  \prod_{j=1}^{\ell}\norm{G_j}_{L^\tau(\R)}.
\end{equation}
For $r_0>0$, set
\[
  G_{j,r_0}(r):=r^\theta F_j(r)\mathbf 1_{[r_0,\infty)}(r),
  \qquad j=1,\ldots,\ell ,
\]
which can be extended by zero to $\R$.  For $x\in[0,r_0/2]$ and $r_j\ge r_0$, $|r_j-x|^{-\theta}\sim r_j^{-\theta}$, which implies
\[
\begin{aligned}
 T(G_{1,r_0},\ldots,G_{\ell,r_0})(x)   \sim \int_{[r_0,\infty)^\ell}
  \prod_{j=1}^{\ell}F_j(r_j)
  \prod_{j=1}^{\ell}|r_j-r_{j+1}|^{-1/\ell}
  \,\dd r_j.
\end{aligned}
\]
Combining this with \eqref{eq:radial-cyclic-correlation-bound}, we have
\[
  \int_{[r_0,\infty)^\ell}
  \prod_{j=1}^{\ell}F_j(r_j)
  \prod_{j=1}^{\ell}|r_j-r_{j+1}|^{-1/\ell}
  \,\dd r_j
  \lesssim
  \prod_{j=1}^{\ell}
  \norm{G_{j,r_0}}_{L^\tau(\R)}
  =
  \prod_{j=1}^{\ell}
  \norm{r^\theta F_j(r)}_{L^\tau([r_0,\infty))},
\]
with the implicit constant independent of $r_0$.  Letting $r_0\rightarrow0$ and using the
monotone convergence theorem yields \eqref{eq:radial-cyclic-weighted-general}.
\end{proof}

\subsection
{Proof of Theorem~\ref{thm:radial-diagonal-schatten}}
\begin{proof}
From the discussion at the beginning of
this section, 
we assume that $V(t,x)$ is radial in
the spatial variable.
By decomposing the real and imaginary parts of $V$ into their positive and negative parts, it suffices to prove
\eqref{eq:radial-diagonal-schatten} for $V\ge0$.
By density, we may further assume that $V$ is bounded and compactly supported.
If $u=u(t,x)$ is spatially radial, we write $u(t,r)$ for its radial profile,
with $r=|x|$. 
We first work at unit frequency $N=1$ and
restore the general frequency scale at the end.

We recall the Bessel functions.  
Let $J_\nu$ and $Y_\nu$ denote the Bessel functions of the first and
second kinds of order $\nu$, respectively; equivalently, they are the
standard linearly independent solutions of the Bessel equation
\[
z^2u''(z)+zu'(z)+(z^2-\nu^2)u(z)=0.
\]  
Recall also that, for $z>0$,
\begin{equation}\label{eq:radial-Bessel-J-definition}
  J_\nu(z)
  :=\left(\frac z2\right)^\nu
    \sum_{k=0}^\infty
    \frac{(-1)^k(z^2/4)^k}{k!\,\Gamma(\nu+k+1)}.
\end{equation}
The Hankel functions of the first and second kinds are then
\begin{equation*}
  H_\nu^{(1)}(z):=J_\nu(z)+iY_\nu(z),
  \qquad
  H_\nu^{(2)}(z):=J_\nu(z)-iY_\nu(z).
\end{equation*}
In particular,
\begin{equation}\label{eq:radial-J-Hankel-relation}
  J_\nu(z)=\frac12\bigl(H_\nu^{(1)}(z)+H_\nu^{(2)}(z)\bigr).
\end{equation}
For further details on Bessel and Hankel functions,
see \cite{WatsonBessel}.  Throughout the proof, set
\[
  \nu:=\frac{d-2}{2},
  \qquad
  \omega_{d-1}:=|\mathbb S^{d-1}|.
\]

\emph{Step 1.  Radial extension operator with unit frequency.}
For $f\in L^2_{\rad}(\R^d)$, its Fourier transform $\widehat f$ is also radial.
We write $\widehat f_{\rad}:(0,\infty)\to\C$ for the radial profile of
$\widehat f$, that is,
\[
  \widehat f(\xi)=\widehat f_{\rad}(|\xi|)
  \qquad\text{for a.e. }\xi\in\R^d.
\]
The radial Hankel formula 
(see, for example, \cite[\S~14.4]{WatsonBessel}) 
gives
\begin{align}
  f(r)
  =r^{-\nu}\int_0^\infty
    \widehat f_{\rad}(\rho)J_\nu(r\rho)\rho^{\nu+1}\,\dd\rho.
  \label{eq:radial-Hankel-inverse}
\end{align}
For $h\in L^2((0,\infty))$, define the radial extension operator
\begin{equation*}
  \mathcal Eh(t,r)
  :=\int_0^\infty
    e^{-it\rho^2}(r\rho)^{1/2}J_\nu(r\rho)
    \chi(\rho)h(\rho)\,\dd\rho.
\end{equation*}
For $f\in L^2_{\rad}(\R^d)$, set
\[
  h_f(\rho)
  :=\omega_{d-1}^{1/2}\rho^{(d-1)/2}
    \widehat f_{\rad}(\rho).
\]
By the radial Plancherel theorem, the correspondence
$f\longmapsto h_f$
is a unitary map from $L^2_{\rad}(\R^d)$ onto
$L^2((0,\infty))$.  Using
\eqref{eq:radial-Hankel-inverse}, \eqref{eq:EN} and
$\widehat{P_1f}_{\rad}(\rho)=\chi(\rho)\widehat f_{\rad}(\rho)$,
we obtain
\begin{equation}\label{eq:radial-E-E1-relation}
  \mathcal Eh_f(t,r)
  =\omega_{d-1}^{1/2}r^{(d-1)/2}E_1f(t,r).
\end{equation}
Hence, for arbitrary $f,g\in L^2_{\rad}(\R^d)$,
\eqref{eq:radial-E-E1-relation} gives
\begin{align*} 
  \ip{\mathcal Eh_f}{M_V\mathcal Eh_g}_{L_t^2(\R;L_r^2(0,\infty))} 
  &=\omega_{d-1}\int_\R\int_0^\infty 
    V(t,r)E_1f(t,r)\overline{E_1g(t,r)} 
    r^{d-1}\,\dd r\dd t\\ 
  &=\ip{E_1f}{M_VE_1g}_{L_t^2(\R;L_x^2(\R^d))}. 
\end{align*}
Thus $\mathcal E^*M_V\mathcal E$ and
$A_{V,1}^{\rad}=E_1^*M_VE_1$ agree under the unitary correspondence
$f\mapsto h_f$, i.e., 
$\mathcal E^*M_V\mathcal E$ is the unitary conjugate of
$A_{V,1}^{\rad}$, and in particular
\[
  \norm{\mathcal E^*M_V\mathcal E}_{\Sch^{\beta_*'}}
  =\norm{A_{V,1}^{\rad}}_{\Sch^{\beta_*'}}.
\]
Therefore, for \eqref{eq:radial-diagonal-schatten} with $N=1$, it suffices to prove
\begin{equation}\label{eq:radialEScha}
  \norm{\mathcal E^*M_V\mathcal E}_{\Sch^{\beta_*'}}
  \lesssim
  \left(
    \int_{\R}\int_0^\infty
    |V(t,r)|^{q_*'} r^{d-1}\,\dd r\dd t
  \right)^{1/q_*'}\sim \|V(t,x)\|_{L^{q_*'}_{t,x}(\R\times \R^d)}.
\end{equation}

\emph{Step 2. Decomposition of radial extension operator $\mathcal E$ by Bessel functions $J_\nu$.} 
Set
\[
 \mathfrak a_{\pm}
  :=(2\pi)^{-1/2}
    e^{\mp i(\frac{\pi\nu}{2}+\frac{\pi}{4})}.
\]
For $z>0$, define the remainders
\begin{align*}
  \mathfrak b_{+}(z)
  &:=\frac12z^{1/2}e^{-iz}H_\nu^{(1)}(z)-\mathfrak a_{+},\\
  \mathfrak b_{-}(z)
  &:=\frac12z^{1/2}e^{iz}H_\nu^{(2)}(z)-\mathfrak a_{-}.
\end{align*}
By \eqref{eq:radial-J-Hankel-relation}, these definitions give the
exact identity
\begin{equation*}
  z^{1/2}J_\nu(z)
  =\sum_{\sigma\in\{+,-\}}
    e^{i\sigma z}\bigl(\mathfrak a_\sigma+\mathfrak b_\sigma(z)\bigr),
  \qquad z>0.
\end{equation*}
The decay estimates of  $\mathfrak b_\pm$  (see e.g., \cite[Proof of Lemma~4.2]{KillipVisanZhang2008}) says that, for every integer
$\ell\ge0$,
\begin{equation}\label{eq:radial-Bessel-symbol}
  |\partial_z^\ell\mathfrak b_{\pm}(z)|
  \lesssim_{\nu,\ell}z^{-1-\ell},
  \qquad z\ge1.
\end{equation}
Choose nonnegative functions
$\eta_0,\eta_\infty\in C^\infty([0,\infty))$ such that
\[
  \eta_0(r)+\eta_\infty(r)=1,
  \qquad
  \eta_0(r)=1\quad(0\le r\le4),
  \qquad
  \supp\eta_0\subset[0,8].
\]
Thus $\eta_\infty(r)=0$ for $r\le4$.  We decompose
\begin{equation}\label{eq:radial-E-decomposition}
  \mathcal E=\mathcal E_<+ \mathcal E_++\mathcal E_-,
\end{equation}
where
\begin{align}
  \mathcal E_<h(t,r)
  &:=\eta_0(r)\int_0^\infty
    e^{-it\rho^2}(r\rho)^{1/2}J_\nu(r\rho)
    \chi(\rho)h(\rho)\,\dd\rho,
  \label{eq:E<}\\
  \mathcal E_\pm h(t,r)
  &:=\eta_\infty(r)
    \int_0^\infty
    e^{\pm ir\rho-it\rho^2}
    \bigl(\mathfrak a_\pm+\mathfrak b_\pm(r\rho)\bigr)\chi(\rho)h(\rho)\,\dd\rho.
  \label{eq:radial-large-Epm}
\end{align}

\emph{Step 3. The exterior terms $\mathcal E_\pm$.}
Fix a compact interval $I\subset(0,\infty)$ whose interior contains
$\supp\chi$.  
For $\sigma\in\{+,-\}$, set
\[
  X_\sigma:=M_{V^{1/2}}\mathcal E_\sigma.
\]
For $r,y>4$, define 
\[
  h_{r,y}^\sigma(\rho)
  := \bigl(\mathfrak a_\sigma+\mathfrak b_\sigma(r\rho)\bigr)\overline{\bigl(\mathfrak a_\sigma+\mathfrak b_\sigma(y\rho)\bigr)}|\chi(\rho)|^2,
  \qquad \rho\in I.
\]
By \eqref{eq:radial-Bessel-symbol} and $\supp\chi\subset (1/2,2)$,
\begin{equation}\label{eq:radial-large-product-amplitude-uniform}
  \sup_{\sigma\in\{+,-\}}\sup_{r,y>4}
  \left(
    \norm{h_{r,y}^\sigma}_{L^\infty(I)}
    +\norm{\partial_\rho h_{r,y}^\sigma}_{L^\infty(I)}
  \right)<\infty.
\end{equation}
For $t\in\R$ and $f\in L^2(\R)$, set
\[
  \kappa_{r,y}^{\sigma}(t)
  :=\int_I
    e^{i\sigma(r-y)\rho-it\rho^2}
    h_{r,y}^\sigma(\rho)\,\dd\rho,
  \qquad \mathcal C_{r,y}^{\sigma}f= \kappa_{r,y}^{\sigma} *f,
\]
\[\text{and}\qquad
  D_{r,y}^{\sigma}
  :=\eta_\infty(r)\eta_\infty(y)
    M_{V(\cdot,r)^{1/2}}
    \mathcal C_{r,y}^{\sigma}
    M_{V(\cdot,y)^{1/2}}.
\]
For $f\in L^2((4,\infty)\times\R)$, a direct computation using \eqref{eq:radial-large-Epm} gives
\[
  (X_\sigma X_\sigma^*f)(t,r)
  =\int_4^\infty D_{r,y}^\sigma (f(\cdot,y))(t)\,\dd y.
\]
Thus $D_{r,y}^\sigma$, as a vector-valued function of $r,y$, is the  kernel of
$X_\sigma X_\sigma^*$ with respect to the radial variable, in the sense of Lemma~\ref{lem:operator-valued-cyclic-trace} \eqref{eq:Bj}.
Set, for $r>4$,
\[
  F(r):=\norm{V(t,r)}_{L_t^{q_*'}(\R)},
\]
and extend $F$ by zero to $0<r\le4$.  Since $\beta_*'=2q_*'$,
\[
  \norm{V(t,r)^{1/2}}_{L_t^{\beta_*'}(\R)}=F(r)^{1/2}.
\]
Hence, using Lemma~\ref{lem:radial-time-schatten} with $h_\alpha=h_{r,y}^\sigma, z=\sigma(r-y), a(t)=V(t,r)^{1/2}, b(t)=V(t,y)^{1/2}$, $\gamma=\beta_*'$
and \eqref{eq:radial-large-product-amplitude-uniform}, we obtain
\begin{equation}\label{eq:radial-large-kernel-bound}
  \norm{D_{r,y}^{\sigma}}_{\Sch^{\beta_*'}(L_t^2(\R))}
  \lesssim
  \langle r-y\rangle^{-1/\beta_*'}F(r)^{1/2}F(y)^{1/2}
  \lesssim
  |r-y|^{-1/\beta_*'}F(r)^{1/2}F(y)^{1/2},
  \quad r\ne y.
\end{equation}
The kernel of
$D_{r,y}^{\sigma}\in \Sch^2(L_t^2(\R))$ is
\[
  \eta_\infty(r)\eta_\infty(y)
  V(t,r)^{1/2}\kappa_{r,y}^{\sigma}(t-s)
  V(s,y)^{1/2},
\]
which is jointly measurable in $(r,y,t,s)$.  Hence the map
$(r,y)\mapsto D_{r,y}^{\sigma}$ is weakly measurable as an
$\Sch^2(L_t^2(\R))$-valued map.
Since $\Sch^2(L_t^2(\R))$ is separable, 
the Pettis measurability theorem (see \cite[Chapter II, Theorem~2]{DiestelUhl1977}) implies that this map is strongly measurable. 
Moreover, since $V$ is bounded and compactly supported and
$\sup_{r,y>4}\norm{\kappa_{r,y}^{\sigma}}_{L_t^\infty(\R)}<\infty$,
one has
\begin{align*}
  \norm{D_{r,y}^{\sigma}}_{\Sch^2(L_t^2(\R))}^2
  &=\eta_\infty(r)^2\eta_\infty(y)^2
    \iint_{\R^2}V(t,r)
      |\kappa_{r,y}^\sigma(t-s)|^2
      V(s,y)\,\dd t\dd s \\
  &\lesssim
  \norm{V(t,r)}_{L_t^1(\R)}
  \norm{V(t,y)}_{L_t^1(\R)}
  \in L^1_{r,y}((4,\infty)^2).
\end{align*}
Hence the operator-valued map
$(r,y)\longmapsto D_{r,y}^\sigma$
belongs to
$L^2_{r,y}\bigl((4,\infty)^2;\Sch^2(L_t^2(\R))\bigr).$
Applying Lemma~\ref{lem:operator-valued-cyclic-trace} with
$X=(4,\infty)$, $H=L_t^2(\R)$, $\ell=\beta_*'=2d+1$, and
$D_j(r,y)\equiv D_{r,y}^{\sigma}$, $\forall\, j=1,\ldots, \ell$, therefore gives
\[
  \Tr\bigl((X_\sigma X_\sigma^*)^{\beta_*'}\bigr)
  =\int_{(4,\infty)^{\beta_*'}}
    \Tr_{L_t^2(\R)}\bigl(
      D_{r_1,r_2}^{\sigma}D_{r_2,r_3}^{\sigma}\cdots
      D_{r_{\beta_*'},r_1}^{\sigma}
    \bigr)
    \prod_{j=1}^{\beta_*'}\dd r_j.
\]
By the Schatten H\"older inequality and
\eqref{eq:radial-large-kernel-bound},
for almost every point in $(4,\infty)^{\beta_*'}$,
\begin{align*}
  \Tr\bigl((X_\sigma X_\sigma^*)^{\beta_*'}\bigr)
  &\lesssim
  \int_{(4,\infty)^{\beta_*'}}
    \prod_{j=1}^{\beta_*'}F(r_j)
    \prod_{j=1}^{\beta_*'}|r_j-r_{j+1}|^{-1/\beta_*'}
    \prod_{j=1}^{\beta_*'}\dd r_j,
\end{align*}
with convention $r_{\beta_*'+1}:=r_1$.  
Since $d\ge2$,
one has
$\beta_*'/(\beta_*'-1)<q_*'\le\beta_*'$.  Applying
Lemma~\ref{lem:radial-cyclic-weighted} with
$\ell=\beta_*'$, $\tau=q_*'$, and $\theta =1-\frac1{\beta_*'}-\frac1{q_*'}=\frac{2(d-1)}{2d+1}$,  to the functions
$F_j\equiv F$
gives
\[
  \Tr\bigl((X_\sigma X_\sigma^*)^{\beta_*'}\bigr)
  \lesssim
  \norm{r^{\theta}F(r)}_{L^{q_*'}((0,\infty))}^{\beta_*'}
  =
  \left(
    \int_4^\infty r^{d-1}
    \int_\R|V(t,r)|^{q_*'}\,\dd t\dd r
  \right)^{\beta_*'/q_*'}.
\]
Consequently, uniformly in $\sigma\in\{+,-\}$,
\begin{equation}\label{eq:radial-large-sign-bound}
  \norm{X_\sigma}_{\Sch^{2\beta_*'}}^2
  = \Tr\bigl((X_\sigma X_\sigma^*)^{\beta_*'}\bigr)^{1/\beta_*'}
  \lesssim 
  \left(
    \int_\R\int_0^\infty
    |V(t,r)|^{q_*'}r^{d-1}\,\dd r\dd t
  \right)^{1/q_*'}.
\end{equation}

\emph{Step 4. The local term $\mathcal E_<$.}
Write \eqref{eq:E<} as
\begin{equation*}
  \mathcal E_<h(t,r)
  =\eta_0(r)r^{\nu+1/2}T_rh(t),
  \qquad 0\le r\le8,
\end{equation*}
where
\[
  T_rh(t)
  :=r^{-\nu} \int_0^\infty e^{-it\rho^2}
    \rho^{1/2}J_\nu(r\rho)
    \chi(\rho)h(\rho)\,\dd\rho.
\]
At $r=0$ the amplitude is understood by continuous extension.  Indeed,
the power-series expansion \eqref{eq:radial-Bessel-J-definition}, together
with $\supp\chi\subset(1/2,2)$, shows that
\[
  (r,\rho)\longmapsto
  r^{-\nu}\rho^{1/2}J_\nu(r\rho)\chi(\rho)
\]
extends smoothly to $[0,8]\times\supp\chi$ and is uniformly bounded there.
For $\mu>0$  set
\[
  (U_0h)(\mu)
  :=(2\sqrt\mu)^{-1/2}h(\sqrt\mu), \qquad \alpha_r(\mu)
  := \sqrt{\pi}\,r^{-\nu}
  J_\nu(r\sqrt{\mu})\chi(\sqrt{\mu}).
\]
where at $r=0$ the second expression is again understood by continuous extension.
Then $U_0$ is unitary on $L^2((0,\infty))$.
Let $\mathcal F_0:L^2((0,\infty))\to L^2(\R)$ denote the
restriction of the one-dimensional Fourier transform to the positive half-line:
\[
  (\mathcal F_0g)(t)
  :=(2\pi)^{-1/2}\int_0^\infty e^{-it\mu}g(\mu)\,\dd\mu.
\]
The change of variables $\mu=\rho^2$
gives the exact factorization
\[
  T_r=\mathcal F_0M_{\alpha_r}U_0.
\]
We claim that, uniformly for $r\in[0,8]$,
\begin{equation}\label{eq:radial-small-fixed-r}
  \norm{T_r^*M_WT_r}_{\Sch^{q_*'}}
  \lesssim \norm{W}_{L^{q_*'}(\R)},
  \qquad W\in L^{q_*'}(\R), \quad W\geq 0.
\end{equation}
For arbitrary $p\ge2$, the Kato--Seiler--Simon inequality (see
\cite[Theorem~4.1]{SimonTrace}), applied to the zero extension of $g\in L^2((0,\infty))$ to $\R$, gives
\begin{equation*}
  \norm{M_f\mathcal F_0M_g}_{\Sch^p}
  \lesssim_p \norm{f}_{L^p(\R)}
  \norm{g}_{L^p((0,\infty))}.
\end{equation*}
Because multiplication
by the unitary $U_0$ does not change singular values,
\begin{align*}
\norm{T_r^*M_WT_r}_{\Sch^{q_*'}}
=
  \norm{M_{W^{1/2}}T_r}_{\Sch^{2q_*'}}^2
  &=\norm{M_{W^{1/2}}\mathcal F_0
    M_{\alpha_r}U_0}^2_{\Sch^{2q_*'}}=\norm{M_{W^{1/2}}\mathcal F_0
    M_{\alpha_r}}^2_{\Sch^{2q_*'}}\\
  &\lesssim
    \norm{W}_{L^{q_*'}(\R)}\norm{\alpha_r}^2_{L^{2q_*'}((0,\infty))}
  \lesssim \norm{W}_{L^{q_*'}(\R)}.
\end{align*}
Here the last inequality follows since the functions
$\alpha_r$ are supported in a fixed compact subset of
$(0,\infty)$ and are uniformly bounded.
Since we have the identity
\begin{equation*}
  \mathcal E_<^*M_V\mathcal E_<
  =\int_0^8T_r^*M_{\eta_0(r)^2r^{d-1}V(\cdot,r)}T_r\,\dd r.
\end{equation*}
Combining this with
\eqref{eq:radial-small-fixed-r} and H\"older's inequality,
we obtain
\begin{align*}
  \norm{\mathcal E_<^*M_V\mathcal E_<}_{\Sch^{q_*'}}
  &\le
  \int_0^8
  \norm{
    T_r^*M_{\eta_0(r)^2r^{d-1}V(\cdot,r)}T_r
  }_{\Sch^{q_*'}}\,\dd r\\
  &\lesssim
  \int_0^8 \eta_0(r)^2r^{d-1}
    \norm{V(t,r)}_{L_t^{q_*'}(\R)}\,\dd r\\
  &\lesssim
  \left(
    \int_{\R}\int_0^\infty
    |V(t,r)|^{q_*'}r^{d-1}\,\dd r\dd t
  \right)^{1/q_*'}.
\end{align*}
Set $X_<:=M_{V^{1/2}}\mathcal E_<$.  Since $V\ge0$, $2q_*'=\beta_*'<2\beta_*'$ and $\Sch^{\beta_*'}\subset\Sch^{2\beta_*'}$, we conclude that
\begin{equation}\label{eq:radial-small-component-bound} \norm{X_<}_{\Sch^{2\beta_*'}}^2\le 
  \norm{X_<}_{\Sch^{2q_*'}}^2
  =\norm{\mathcal E_<^*M_V\mathcal E_<}_{\Sch^{q_*'}}
  \lesssim 
  \left(
    \int_{\R}\int_0^\infty
    |V(t,r)|^{q_*'} r^{d-1}\,\dd r\dd t
  \right)^{1/q_*'}.
\end{equation}

\emph{Step 5. Recombination and frequency rescaling.}
By \eqref{eq:radial-E-decomposition},
\[
  M_{V^{1/2}}\mathcal E=X_<+X_++X_-.
\]
The triangle inequality, together with
\eqref{eq:radial-large-sign-bound} and
\eqref{eq:radial-small-component-bound}, yields \eqref{eq:radialEScha} as
\[
  \norm{\mathcal E^*M_V\mathcal E}_{\Sch^{\beta_*'}}
  =\norm{M_{V^{1/2}}\mathcal E}_{\Sch^{2\beta_*'}}^2
  \lesssim
  \norm{X_<}_{\Sch^{2\beta_*'}}^2
  +\norm{X_+}_{\Sch^{2\beta_*'}}^2+\norm{X_-}_{\Sch^{2\beta_*'}}^2
  \lesssim
  \|V(t,x)\|_{L^{q_*'}_{t,x}(\R\times\R^d)}.
\]
The preceding argument proves the desired Schatten estimate \eqref{eq:radial-diagonal-schatten} at unit
frequency  $N=1$.  It remains to recover the estimate at an
arbitrary frequency scale.  This follows from the standard
Schr\"odinger scaling.  For $N>0$, recall the unitary Schr\"odinger scaling
\[
\delta_N: f\mapsto  f_N(x):=N^{d/2}f(Nx).
\]
Then the scaling intertwines with the truncation and Schr\"odinger semigroup by
\[
  P_N\delta_N=\delta_NP_1,
  \qquad
  U_S(t)\delta_N=\delta_N U_S(N^2t). 
\] 
Define the rescaled multiplier $V_N(t,x):=V(N^{-2}t,N^{-1}x)$. Then
\[
  \norm{V_N}_{L_{t,x}^{q_*'}(\R\times\R^d)}
  =N^{(d+2)/q_*'}\norm{V}_{L_{t,x}^{q_*'}(\R\times\R^d)}.
\]
For $f,g\in L^2_{\rad}(\R^d)$, a direct change of variables gives
\[
  \ip{\delta_Nf}{A_{V,N}^{\rad}\delta_Ng}
  =N^{-2}\ip{f}{A_{V_N,1}^{\rad}g},
  \qquad
  \delta_N^*A_{V,N}^{\rad}\delta_N=N^{-2}A_{V_N,1}^{\rad}.
\]
Then the unit-frequency estimate and unitary invariance of Schatten norms yield
\begin{align*}
  \norm{A_{V,N}^{\rad}}_{\Sch^{\beta_*'}}
  =N^{-2}\norm{A_{V_N,1}^{\rad}}_{\Sch^{\beta_*'}}\lesssim N^{-2+(d+2)/q_*'}\norm{V}_{L_{t,x}^{q_*'}(\R\times\R^d)}=N^{-\frac{2(d-1)}{2d+1}}\norm{V}_{L_{t,x}^{q_*'}(\R\times\R^d)}.
\end{align*}
This proves \eqref{eq:radial-diagonal-schatten} for general frequency $N>0$. The proof is complete.
\end{proof}

\subsection{Proof of Theorem~\ref{thm:intro-radial}}

\begin{proof}
The Schatten bound \eqref{eq:radial-diagonal-schatten} in Theorem~\ref{thm:radial-diagonal-schatten}, together with the duality principle Lemma~\ref{thm:intro-abstract} and Remark~\ref{rem:abstract-closed-subspace},  implies the Bessel-family Strichartz estimate 
\eqref{eq:radial-boundary-estimate} at the diagonal point $(p,q)=(q_*,q_*)$, with 
$\beta=\beta_*=\beta(q_*)$; see point $D$ in Figure~\ref{fig:radial-admissible}.
Along the radial boundary $(C,E]$,  that is $p,q$ satisfy \eqref{eq:radial-boundary-segment},
Theorem~\ref{thm:intro-known-radial-single} and
the triangle inequality give the $\ell^1$ bound
\begin{align}\label{eq:radial-proof-trivial-l1-input}
  \norm{\sum_j\lambda_j|U_S(t)P_Nf_j|^2}_{L_t^{p}(\R;L_x^{q}(\R^d))}
  &\le \sum_j|\lambda_j|
  \norm{U_S(t)P_Nf_j}_{L_t^{2p}(\R;L_x^{2q}(\R^d))}^2\nonumber\\
&  \lesssim N^{-(d-1)(1-1/q)}\norm{\lambda}_{\ell^1}.
\end{align}
Interpolating between \eqref{eq:radial-boundary-estimate} at the diagonal point $D: q=q_*$,  
and the $\ell^1$ bound \eqref{eq:radial-proof-trivial-l1-input} at the trivial point $E: q=1$ and at arbitrary point near the endpoint $C: q=q_c$, respectively, 
we obtain
\eqref{eq:radial-boundary-estimate} for $1\le q\le q_*$ with
$\beta=\beta(q)$, and 
 for $q_*< q< q_c$ with
$1\le \beta<2p/(p+1)$. The proof is complete.
\end{proof}

\subsection{Necessary conditions} 

\begin{proposition}
\label{prop:radial-necessary}
Let $d\ge2$, let $1\le p\le\infty$, $1\le q<\infty$, and $\beta\ge1$.
Suppose that \eqref{eq:radial-boundary-estimate} holds
for every radial Bessel family
$f=(f_j)_j\subset L^2_{\rad}(\R^d)$ and every sequence $\lambda=(\lambda_j)_j\in\ell^\beta$.
Then
\begin{equation}\label{eq:radial-shell-necessary-condition}
  \frac1\beta
  \ge
  \frac1p+\frac d q-(d-1).
\end{equation}
If in addition $p<\infty$, then
\begin{equation}\label{eq:radial-beta-p-necessary}
  \beta\le p.
\end{equation}
\end{proposition}

\begin{proof}
It suffices to test
\eqref{eq:radial-boundary-estimate} on radial orthonormal
systems and $N=1$.  Write $\omega_{d-1}:=|\mathbb S^{d-1}|$.

\emph{Step 1: The time-translation condition $\beta\le p$.}
Assume that $p<\infty$.  
For radial functions
$f,g\in L^2(\R^d)$, radial Plancherel gives
\begin{equation}
\label{eq:radial-plancherel-inner-product}
  \ip{f}{g}_{L^2(\R^d)}
  =\omega_{d-1}\int_0^\infty
  \widehat f_{\rad}(\rho)
  \overline{\widehat g_{\rad}(\rho)}
  \rho^{d-1}\,\dd\rho.
\end{equation}
Choose $T>0$ sufficiently large so that
\[1+\frac{2\pi}{T}\le \frac{16}{9}. 
\]
We prescribe a radial Fourier transform
$\widehat f\ge 0$ for $f\in L^2$ by requiring
\begin{equation}
\label{eq:flat-spectral-profile}
 \frac{\omega_{d-1}}{2}|\widehat f_{\rad}(\sqrt{\mu})|^2\mu^{\frac d2-1}
 =\frac{T}{2\pi}\mathbf 1_{[1,1+2\pi/T]}(\mu).
\end{equation}
By \eqref{eq:radial-plancherel-inner-product} and
\eqref{eq:flat-spectral-profile},
\[
 \norm{f}_{L^2(\R^d)}^2
 =\frac{T}{2\pi}\int_1^{1+2\pi/T}\dd\mu
 =1.
\]
Moreover, $\supp\widehat f\subset\{\xi:1\le|\xi|\le\sqrt{1+2\pi/T}\}$,
which is contained in $\{\xi:3/4\le|\xi|\le4/3\}$ by the choice of
$T$.  Hence $P_1f=f$.

For each $j\in\Z$, define $g_j=e^{-ijT\Delta}f$; then every $g_j$ is radial and
$P_1g_j=g_j$.
Using \eqref{eq:radial-plancherel-inner-product} and
\eqref{eq:flat-spectral-profile}, followed by the change of variables
$\mu=\rho^2$, we obtain\[
 \begin{aligned}
 \ip{g_j}{g_k}
 =\frac{\omega_{d-1}}{2}\int_0^\infty
   e^{i(j-k)T\mu}|\widehat f_{\rad}(\sqrt{\mu})|^2
   \mu^{\frac d2-1}\,\dd\mu
 =\frac{T}{2\pi}
   \int_1^{1+2\pi/T}e^{i(j-k)T\mu}\,\dd\mu
 =\delta_{jk}.
 \end{aligned}
\]
Thus $(g_j)_{j\in\Z}$ is a radial orthonormal system.  

Set $u(t,x)=U_S(t)f(x)$.  Since $\widehat f\in L^1(\R^d)$, $u$ is
continuous on $\R\times\R^d$.  Moreover,
\[
 u(0,0)=f(0)
 =(2\pi)^{-d/2}\int_{\R^d}\widehat f(\xi)\,\dd\xi>0.
\]
Thus there exist a number $\delta$ with $0<\delta<T/2$, a ball
$B\subset\R^d$, and a constant $c_0>0$ such that, 
\[
 |u(t,x)|\ge c_0,
 \qquad |t|\le\delta, x\in B.
\]
The group property gives, for every $j\in\Z$,
\[
 U_S(t)P_1g_j(x)
 =U_S(t)g_j(x)
 =e^{i(t-jT)\Delta}f(x)
 =u(t-jT,x).
\]
Therefore, for every positive integer $K$,
\begin{align*}
  \norm{
    \sum_{j=1}^K|U_S(t)P_1g_j|^2
  }_{L_t^p(\R;L_x^q(\R^d))}^p
  &\ge
  \sum_{j=1}^K\int_{jT-\delta}^{jT+\delta}
  \norm{u(t-jT,x)}_{L_x^{2q}(\R^d)}^{2p}\,\dd t\\
  &=K\int_{-\delta}^{\delta}
  \norm{u(t,x)}_{L_x^{2q}(\R^d)}^{2p}\,\dd t
  \gtrsim K.
\end{align*}
On the other hand, applying the estimate
\eqref{eq:radial-boundary-estimate} to the orthonormal system $(g_j)_{j=1}^K$ with $\lambda_j=1$,
we obtain
\[
  K^{1/p}\lesssim \norm{(1,\ldots,1)}_{\ell^\beta}= K^{1/\beta}.
\]
Letting $K\to\infty$ gives \eqref{eq:radial-beta-p-necessary}.

\emph{Step 2: The radial-shell condition \eqref{eq:radial-shell-necessary-condition}.}
Set $\nu:=(d-2)/2$.
Choose a nonnegative function $\eta\in C_c^\infty((-1/100,1/100))$
satisfying $\norm{\eta}_{L^2(\R)}=1$.  For every sufficiently large integer
$M$, set
\[
  \rho_j:=1+\frac{j}{10M},
  \qquad 1\le j\le M,
\]
and
\begin{equation*}
  h_j(\rho)
  =M^{1/2}\eta\bigl(M(\rho-\rho_j)\bigr),
  \qquad 1\le j\le M.
\end{equation*}
Since $\supp\eta\subset(-1/100,1/100)$, the supports of the functions
$h_j$ are pairwise disjoint and contained in $(3/4,4/3)$ for all
sufficiently large $M$, and $\norm{h_j}_{L^2((0,\infty))}=1$.

Define radial functions $g_j$ by
\[
  \widehat g_{j,\rad}(\rho)
  =\omega_{d-1}^{-1/2}\rho^{-\frac{d-1}{2}}h_j(\rho).
\]
The disjointness of the supports of the $h_j$ and \eqref{eq:radial-plancherel-inner-product} give
$\ip{g_j}{g_\ell}=\delta_{j\ell}$.  Thus $(g_j)_{j=1}^{M}$ is a radial orthonormal system.
Moreover, $\supp\widehat g_{j,\rad}\subset(3/4,4/3)$, so
$P_1g_j=g_j$.
Define
$$u_j(t,|x|)=U_S(t)P_1g_j(x)=U_S(t)g_j(x).$$
By the Hankel formula \eqref{eq:radial-Hankel-inverse},
\begin{equation}\label{eq:radial-shell-hankel-representation}
  u_j(t,r)
  =\omega_{d-1}^{-1/2}r^{-\nu}
  \int_0^\infty
  e^{-it\rho^2}J_\nu(r\rho)
  \rho^{1/2}h_j(\rho)\,\dd\rho.
\end{equation}
The large-argument expansion for the Bessel function, see \cite[\S~7.3]{WatsonBessel},
gives, uniformly for $\rho\in[3/4,4/3]$,
\begin{equation}\label{eq:radial-shell-bessel-asymptotic}
  J_\nu(r\rho)
  =\left(\frac{2}{\pi r\rho}\right)^{1/2}
  \cos\left(r\rho-\frac{\pi\nu}{2}-\frac\pi4\right)
  +O((r\rho)^{-3/2}),
  \qquad r\to\infty.
\end{equation}
Set
$t=Ms$, $r=My$.
On the support of $h_j$, make the change of variables
\[
  \rho=\rho_j+\frac{\tilde{\rho}}{M}.
\]
The two exponentials arising from the cosine in \eqref{eq:radial-shell-bessel-asymptotic} give two components.
More precisely, define
\begin{align*}
  G_{j,M}^+(s,y)
  &:=\int_\R
  e^{i[(y-2s\rho_j)\tilde{\rho}-s\tilde{\rho}^2/M]}
  \eta(\tilde{\rho})\,\dd\tilde{\rho},\\
  G_{j,M}^-(s,y)
  &:=\int_\R
  e^{-i[(y+2s\rho_j)\tilde{\rho}+s\tilde{\rho}^2/M]}
  \eta(\tilde{\rho})\,\dd\tilde{\rho}.
\end{align*}
Then
\eqref{eq:radial-shell-hankel-representation} and
\eqref{eq:radial-shell-bessel-asymptotic} yield
\begin{align}\label{eq:radial-shell-rescaled-solution}
  u_j(Ms,My)
  ={}&(2\pi\omega_{d-1})^{-1/2}M^{-d/2}y^{-\frac{d-1}{2}}
  \Bigl[
    e^{i(My\rho_j-Ms\rho_j^2-\frac{\pi\nu}{2}-\frac\pi4)}
    G_{j,M}^+(s,y) \\
    & +e^{-i(My\rho_j+Ms\rho_j^2-\frac{\pi\nu}{2}-\frac\pi4)}
    G_{j,M}^-(s,y)
  \Bigr]\notag
  \quad+\mathcal R_{j,M}(s,y),
\end{align}
where $\mathcal R_{j,M}$ denotes the contribution of the error term in
\eqref{eq:radial-shell-bessel-asymptotic}.

Let
\[
  \widetilde{\eta}(z)
  :=\int_\R e^{iz\tilde{\rho}}\eta(\tilde{\rho})\,\dd\tilde{\rho}.
\]
Since $\eta$ is compactly supported, uniformly for $1\le\rho_j\le11/10$, bounded $s$, and $y\in\R$,
\begin{align}
  G_{j,M}^+(s,y)
  &\longrightarrow
  \widetilde\eta(y-2s\rho_j),
  \label{eq:radial-outgoing-limit}\\
  G_{j,M}^-(s,y)
  &\longrightarrow
  \widetilde\eta(-y-2s\rho_j)
  \label{eq:radial-incoming-limit}
\end{align}
as $M\to\infty$.
Choose $R_1>0$ so that
\begin{equation}\label{eq:radial-profile-central-mass}
  c_\eta
  :=\int_{-R_1}^{R_1}|\widetilde{\eta}(z)|^2\,\dd z>0.
\end{equation}
Since $\widetilde{\eta}$ is a Schwartz function, we can then choose $R_2\gg R_1$ such that
\begin{equation}\label{eq:radial-incoming-tail-small}
  \int_{-\infty}^{-R_2}
  |\widetilde{\eta}(z)|^2\,\dd z
  \le
  \frac{c_\eta}{100}.
\end{equation}
Set
\[
  a=2R_2-R_1>0,
  \qquad
  b=\frac{11}{5}(R_2+1)+R_1.
\]
For every $s\in[R_2,R_2+1]$ and $1\le j\le M$,
$y-2s\rho_j\in[-R_1,R_1]$ implies $y\in[a,b]$,
while $y\in[a,b]$ implies $-y-2s\rho_j\le -R_2$.
Therefore, \eqref{eq:radial-profile-central-mass},
\eqref{eq:radial-incoming-tail-small}, and the uniform convergences
\eqref{eq:radial-outgoing-limit}--\eqref{eq:radial-incoming-limit}
show that, for all sufficiently large $M$,
\[
  \norm{G_{j,M}^+(s,y)}_{L_y^2([a,b])}
  \ge \frac34c_\eta^{1/2},
  \qquad
  \norm{G_{j,M}^-(s,y)}_{L_y^2([a,b])}
  \le \frac14c_\eta^{1/2},
\]
uniformly for $1\le j\le M$ and $s\in[R_2,R_2+1]$.
By the reverse triangle inequality,
\begin{equation}\label{eq:radial-main-profile-lower}
  \begin{aligned}
  &\norm{
    e^{i(My\rho_j-Ms\rho_j^2-\frac{\pi\nu}{2}-\frac\pi4)}
    G_{j,M}^+(s,y)
    +e^{-i(My\rho_j+Ms\rho_j^2-\frac{\pi\nu}{2}-\frac\pi4)}
    G_{j,M}^-(s,y)
  }_{L_y^2([a,b])}\\
  &\qquad\ge
  \norm{G_{j,M}^+(s,\cdot)}_{L_y^2([a,b])}
  -\norm{G_{j,M}^-(s,\cdot)}_{L_y^2([a,b])}
  \ge \frac12c_\eta^{1/2}.
  \end{aligned}
\end{equation}

It remains to control the remainder in
\eqref{eq:radial-shell-rescaled-solution}.
Indeed, the error term in
\eqref{eq:radial-shell-bessel-asymptotic} gives
\[
  |\mathcal R_{j,M}(s,y)|
  \lesssim
  r^{-\nu}
  \int_{3/4}^{4/3} (r\rho)^{-3/2}\rho^{1/2}|h_j(\rho)|\,\dd\rho
  \lesssim
  r^{-\frac{d+1}{2}}\norm{h_j}_{L^1((0,\infty))}.
\]
Since
$\norm{h_j}_{L^1((0,\infty))}=M^{-1/2}\norm{\eta}_{L^1(\R)}$
and $r=My$, uniformly for $1\le j\le M$, $s\in[R_2,R_2+1]$, we obtain
\begin{equation}\label{eq:radial-shell-remainder-small}
  \int_a^b
  |\mathcal R_{j,M}(s,y)|^2
  M^dy^{d-1}\,\dd y
  \lesssim M^{-2} \rightarrow 0
  \quad\text{as }M\to\infty.
\end{equation}

Combining \eqref{eq:radial-shell-rescaled-solution},
\eqref{eq:radial-main-profile-lower}, and
\eqref{eq:radial-shell-remainder-small}, we conclude that, for all
sufficiently large $M$,
\begin{equation}\label{eq:radial-single-shell-mass}
  \int_{aM}^{bM}
  |u_j(t,r)|^2r^{d-1}\,\dd r
  \ge c_1
\end{equation}
whenever
$R_2M\le t\le(R_2+1)M$, $1\le j\le M$,
where $c_1>0$ is independent of $M$, $t$, and $j$.
Set
\[
  F_M(t,x)=\sum_{j=1}^{M}|U_S(t)P_1g_j(x)|^2,
  \qquad
  \Omega_M=\{x\in\R^d:aM\le|x|\le bM\}.
\]
By \eqref{eq:radial-single-shell-mass}, for every
$t\in[R_2M,(R_2+1)M]$,
\[
  \int_{\Omega_M}F_M(t,x)\,\dd x
  \gtrsim M.
\]
Since $|\Omega_M|\sim M^d$, H\"older's inequality gives
\[
  \norm{F_M(t,x)}_{L_x^q(\R^d)}
  \ge
  |\Omega_M|^{1/q-1}
  \int_{\Omega_M}F_M(t,x)\,\dd x
  \gtrsim M^{\frac d q-(d-1)}.
\]
Hence
\begin{equation}\label{eq:radial-family-mixed-lower}
  \norm{F_M}_{L_t^p(\R;L_x^q(\R^d))}
  \gtrsim
  M^{\frac1p+\frac d q-(d-1)}.
\end{equation}
Applying \eqref{eq:radial-boundary-estimate} to the orthonormal
system $(g_j)_{j=1}^{M}$ with
$\lambda_j=1$, and
combining with \eqref{eq:radial-family-mixed-lower}, we obtain
\[
  M^{\frac1p+\frac d q-(d-1)}
  \lesssim M^{1/\beta}.
\]
Letting $M\to\infty$ proves
\eqref{eq:radial-shell-necessary-condition}.
\end{proof}

\begin{remark}
\label{rem:radial-global-necessary}
The necessary conditions in Proposition~\ref{prop:radial-necessary} also
apply to the corresponding frequency-global estimate.
Moreover, for the Schr\"odinger scaling line $2/p+d/q=d$, corresponding to the segment $EF$ in
Figure~\ref{fig:radial-admissible},
on the subsegment $BF$,
the range $\beta<p$ is known for the orthonormal
Strichartz estimate without radial restriction.
On the other hand, Proposition~\ref{prop:radial-necessary} yields the
necessary condition $\beta\le p$ even for radial datum.
Thus the range $\beta<p$ is sharp in the radial setting as well, in the
sense that no such estimate can hold for $\beta>p$.
\end{remark}

\medskip
\noindent\textbf{Acknowledgements.}
The work is partially supported by the Beijing Natural Science Foundation
(No.1242009), the National Key R\&D Program of China (No.2024YFA1015300),
and the National Natural Science Foundation of China
(Nos.11801536,12101028).  A.Z.~wishes to thank Zihua Guo for stimulating conversations and for generously sharing his ideas, and to thank Monash University for its hospitality during his visit.

\noindent\textbf{Declarations.}
The authors have no relevant financial or non-financial interests to
disclose.  Data sharing is not applicable because no datasets were generated
or analyzed in this study.

\bibliographystyle{alpha}
\bibliography{ref}

\end{document}